\documentclass{amsart}
\usepackage{amsmath, amsthm, amsfonts}

\newtheorem{thm}{Theorem}[section]

\newtheorem{prop}[thm]{Proposition}
\newtheorem{cor}[thm]{Corollary}

\theoremstyle{definition}

\newtheorem{defi}[thm]{Definition}
\newtheorem{exe}[thm]{Example}
\usepackage[cp1250]{inputenc}
\usepackage{amssymb}
\usepackage[T1]{fontenc}

\usepackage{graphicx}
\usepackage{amssymb}
\usepackage{times}
\usepackage{csquotes}
\usepackage{yfonts}
\usepackage{bbm}
\usepackage{dsfont}
\usepackage{enumitem}
\usepackage{bbm}
\usepackage{xcolor}
\usepackage{url}
\usepackage{graphicx}
\usepackage{float}

\title[coherent distributions on the square]{coherent distributions on the square \\ -- Extreme points and asymptotics }

\author{Stanis\l{}aw Cichomski}
\address{Department of Mathematics, Informatics and Mechanics\\
 University of Warsaw\\
Banacha 2, 02-097 Warsaw\\
Poland}

\author{Adam Os\k{e}kowski}
\address{Department of Mathematics, Informatics and Mechanics\\
 University of Warsaw\\
Banacha 2, 02-097 Warsaw\\
Poland}

\numberwithin{equation}{section}

\begin{document}

\begin{abstract} Let $\mathcal{C}$ denote the family of all coherent distributions on the unit square $[0,1]^2$, i.e. all those probability measures $\mu$ for which there exists a random vector $(X,Y)\sim \mu$, a pair $(\mathcal{G},\mathcal{H})$ of $\sigma$-fields and an event $E$ such that $X=\mathbb{P}(E|\mathcal{G})$, $Y=\mathbb{P}(E|\mathcal{H})$ almost surely. In this paper we examine the set $\mathrm{ext}(\mathcal{C})$  of extreme points of $\mathcal{C}$ and provide its general characterisation. Moreover, we establish several structural properties of finitely-supported elements of $\mathrm{ext}(\mathcal{C})$. We apply these results to obtain the asymptotic sharp bound
$$\lim_{\alpha \to \infty} \alpha\cdot \Big(\sup_{(X,Y)\in \mathcal{C}}\mathbb{E}|X-Y|^{\alpha}\Big)  =  \frac{2}{e}.$$ \end{abstract}

\maketitle

\section{Introduction} 
Let $\mu$ be a probability measure on the unit square $[0,1]^2$. Following \cite{C1}, this measure is called \emph{coherent}, if it is the joint distribution of a two-variate random vector $(X,Y)$ defined on some arbitrary probability space $(\Omega, \mathcal{F}, \mathbb{P})$, such that 
$$X=\mathbb{P}(E|\mathcal{G}) \ \ \ \ \text{and} \ \ \ \ Y=\mathbb{P}(E|\mathcal{H}), \ \ \ \  \text{almost surely,}$$
for some measurable event $E\in \mathcal{F}$ and some two sub-$\sigma$-fields $\mathcal{G}, \mathcal{H} \subset \mathcal{F}$. Throughout the text, the class of all coherent probability measures will be denoted by $\mathcal{C}$; for the sake of convenience (and with a slight abuse of notation), we will also write $(X, Y) \in \mathcal{C}$ to indicate that the distribution of a random vector $(X, Y)$ is coherent. 

Coherent measures enjoy the following nice interpretation. Suppose that two experts provide their personal estimates on the likelihood of some random event $E$, and assume that the knowledge of the first and the second expert is represented by the $\sigma$-algebras $\mathcal{G}$ and $\mathcal{H}$, respectively. Then a natural idea to model the predictions of the experts is to use conditional expectations: this leads to the random variables $X$ and $Y$ as above.

 The importance of coherent distributions stem from their numerous applications in statistics (cf. \cite{C1, C2, C4, C3}) and economics  (consult \cite{B4, B2, B1, B3}). 
Coherent distributions are also closely related to graph theory and combinatorial matrix theory, see for instance \cite{GaleRyser, mastersthesis, BPC, tao}.
Moreover, there has been a substantial purely probabilistic advancement on this subject during the last decade, see \cite{contra, pitman, EJP, contra2, kDoob, zhu}. The main interest, both in applied and theoretical considerations, involves bounding the maximal discrepancy of coherent vectors measured by different functionals. A canonical result of this type is the following threshold bound of Burdzy and Pal \cite{contra}.

\begin{thm} \label{contr} For any parameter $\delta \in (\frac{1}{2},1]$, we have
\begin{equation} \label{Eq1}\sup_{(X,Y)\in \mathcal{C}} \mathbb{P}(|X-Y|\ge \delta) = \frac{2(1-\delta)}{2-\delta}. \end{equation} \end{thm}

For a generalisation of (\ref{Eq1}) to $n$-variate coherent vectors, consult \cite{contra2}. Another important example is the expectation bound established independently  in \cite{B1, mastersthesis}.

\begin{thm} For any exponent $\alpha\in (0,2]$, we have 
\begin{equation} \label{Eq2}\sup_{(X,Y)\in \mathcal{C}}\mathbb{E}|X-Y|^{\alpha}   =   2^{-\alpha}.\end{equation}
\end{thm} 

The analysis of the left-hand side of (\ref{Eq2}) for $\alpha>2$ remains a major open problem and constitutes one of the main motivations for this paper. Accordingly, we investigate the asymptotic behavior of this expression and derive an appropriate sharp estimate.

\begin{thm} \label{2/e} We have
\begin{equation} \label{asym} \lim_{\alpha \to \infty} \alpha\cdot \Big(\sup_{(X,Y)\in \mathcal{C}}\mathbb{E}|X-Y|^{\alpha}\Big)  =  \frac{2}{e}.\end{equation}
\end{thm}

The proof of (\ref{asym})  that we present below rests on a novel, geometric-type approach. As verified in \cite{pitman}, the  family of coherent distributions 
 is a convex, compact subset of the space of probability distributions on $[0,1]^2$ equipped with the usual weak topology. One of the main results of this paper is to provide a characterisation of the extremal points of $\mathcal{C}$, which is considered to be one of the major challenges of the topic \cite{pitman,zhu}. 


%

It is instructive to take a look at the corresponding problem arising in the theory of martingales, the solution to which is well-known. Namely  (see \cite{extmart}), fix $N\in \mathbb{N}$ and consider the class of all finite martingales $(M_1,M_2,\dots, M_N)$ and the induced distributions on $\mathbb R^N$. The extremal distributions can be characterised as follows:
\begin{enumerate}[label=(\roman*)]
\item $M_1$ is concentrated in one point,
\item for any $n=2,\,3,\,\ldots,\,N$, the conditional distribution of $M_n$ given $(M_i)_{i=1}^{n-1}$ is concentrated on the set of cardinality at most two. 
\end{enumerate}

In particular, the support of a two-variate martingale  with an extremal distribution cannot exceed two points. Surprisingly, the structure of 
$\mathrm{ext}(\mathcal{C})$ (the set of extreme points of $\mathcal{C}$) is much more complex, as there exist extremal coherent measures with arbitrary large  or even countable infinite  number of atoms (see \cite{B1,zhu}). Conversely, as proved in \cite{B1}, elements of $\mathrm{ext}(\mathcal{C})$ are
always supported on sets of  Lebesgue measure zero. The existence of  non-atomic extreme points remains a yet another open problem.

 For the further discussion, we need to introduce some additional background and notation. For a  measure $\mu$ supported on $[0,1]^2$, we will write  $\mu^x$ and $\mu^y$ for the marginal measures of $\mu$ on $[0,1]$, i.e. for the measures obtained by projecting $\mu$ on the first and the second coordinate, correspondingly. 

\begin{defi}\label{R-set}
Introduce the family $\mathcal{R}$, which consists of all ordered pairs $(\mu, \nu)$ of nonnegative Borel measures on $[0,1]^2$ for which
$$\int_{A}(1-x) \ \mathrm{d}\mu^x  \  \ = \  \ \int_{A} x \ \mathrm{d}\nu^x,$$
and
$$\int_{B}(1-y)  \ \mathrm{d}\mu^y  \ \ =  \ \ \int_{B} y \ \mathrm{d}\nu^y,$$
for any Borel subsets $A,B \in \mathcal{B}([0,1])$. 
\end{defi}

It turns out that the  family $\mathcal{R}$ is very closely related to the class of coherent distributions. We will prove the following statement (a slightly different formulation can be found in \cite{B1}).

\begin{prop} \label{m=mu+nu} Let $m$ be a probability measure on $[0,1]^2$. Then $m$ is coherent if and only if there exists $(\mu, \nu)\in \mathcal{R}$ such that $m=\mu+\nu$. 
\end{prop}

The above result motivates the following.

\begin{defi}
For a fixed $m\in \mathcal{C}$, consider the class
$$\mathcal{R}(m) \ = \ \{(\mu, \nu) \in \mathcal{R} : \ m=\mu+\nu\}.$$
Any element $(\mu, \nu)\in \mathcal{R}(m)$ will be called a \emph{representation} of a coherent distribution $m$.
 \end{defi}

By the very definition, both $\mathcal{C}$ and $\mathcal{R}$, and hence also $\mathcal{R}(m)$, are convex sets. To proceed, let us distinguish the ordering in the class of measures, which will often be used in our considerations below. Namely, for two Borel measures $\mu_1, \mu_2$ supported on the unit square, we will write $\mu_1\leq \mu_2$ if we have $\mu_1(A)\le \mu_2(A)$ for all $A\in \mathcal{B}([0,1]^2)$.

\begin{defi} \label{Un-Mi} Let $m\in \mathcal{C}$.  We say that the representation $(\mu, \nu)$ of $m$ is
\smallskip

$\cdot$ \emph{unique}, if for every $(\tilde{\mu}, \tilde{\nu})\in \mathcal{R}$ with $m=\tilde{\mu}+\tilde{\nu}$, we have $\tilde{\mu}=\mu$ and $\tilde{\nu}=\nu$;

\smallskip

$\cdot$ \emph{minimal}, if for all $(\tilde{\mu}, \tilde{\nu})\in \mathcal{R}$ with $\tilde{\mu}\le \mu$ and $\tilde{\nu}\le \nu$, there exists $\alpha \in [0,1]$ such that $(\tilde{\mu}, \tilde{\nu}) = \alpha \cdot (\mu, \nu)$.
\end{defi}

With these notions at hand, we will give the following general characterisation of $\textrm{ext}(\mathcal{C})$. 

\begin{thm} \label{char} Let $m$ be a coherent distribution on $[0,1]^2$. Then $m$ is extremal if and only if the representation of $m$ is unique and minimal.\end{thm}

This statement will be established in the next section. Then, in Section 3, we concentrate on extremal coherent  measures with finite support. Let $\textrm{ext}_f(\mathcal{C}) = \{\eta \in \textrm{ext}(\mathcal{C}) :  |\mathrm{supp}(\eta)|<\infty\}.$ Theorem \ref{char} will enable us to deduce several structural properties of $\textrm{ext}_f(\mathcal{C})$; most importantly,  as conjectured in \cite{zhu}, we show that support of  $\eta \in \textrm{ext}_f(\mathcal{C})$ cannot contain any axial cycles. Here is the definition.
 
\begin{defi} \label{ax-cycle} The sequence  $\big((x_i,y_i)\big)_{i=1}^{2n}$ with values in $[0,1]^2$ is called an \emph{axial cycle}, if all points
$(x_i, y_i)$ are distinct, the endpoint coordinates $x_1$ and $x_{2n}$ coincide, and we have
\begin{center} $x_{2i} = x_{2i+1}$ \ \ \ and  \ \ \ $y_{2i-1} = y_{2i}$  \ \ \ for all  \ $i$. \end{center}
\end{defi}

Remarkably, the same `no axial cycle' property holds true for extremal doubly stochastic measures (permutons) -- for the relevant discussion, see \cite{EDSM}. Next, in Section 4, we
apply our previous results and obtain the following reduction towards Theorem \ref{2/e}. Namely, for all $\alpha \ge1$, we have
\begin{equation} \label{RED} \sup_{(X,Y)\in \mathcal{C}}\mathbb{E}|X-Y|^{\alpha}  \ \ =  \ \ \sup_{\tilde{\textbf{z}}}\ \sum_{i=1}^{n} z_i  \Big| \frac{z_i}{z_{i-1}+z_i}-\frac{z_i}{z_i+z_{i+1}} \Big|^{\alpha}.\end{equation}
Here the supremum is taken over all $n$ and all sequences \ $\tilde{\textbf{z}} = (z_0, z_1, \dots, z_{n+1})$ such that $  z_0=z_{n+1}=0$, $z_i> 0$  for all $i=1,\,2,\,\ldots,\,n$,  and $\sum_{i=1}^{n}z_i=1.$ 
Finally, using several combinatorial arguments and reductions, we prove Theorem \ref{2/e} by a direct analysis of  the right-hand side of (\ref{RED}).

\section{Coherent measures, Representations}

Let $\mathcal{M}([0,1]^2)$ and $\mathcal{M}([0,1])$ denote  the space of nonnegative Borel measures on $[0,1]^2$ and $[0,1]$, respectively.
For $\mu \in  \mathcal{M}([0,1]^2)$, let $\mu^x, \mu^y \in \mathcal{M}([0,1])$ be defined by

\begin{center}$\mu^x(A) = \mu(A\times [0,1])$ \ \ \ \ and  \ \ \ \  $\mu^y(B)=\mu([0,1]\times B)$,\end{center}
 for all Borel subsets $A,B\in \mathcal{B}([0,1])$. We begin with the following characterisation of $\mathcal{C}$.

\begin{prop} \label{def2} Let $m \in \mathcal{M}([0,1]^2)$. The measure $m$ is a coherent distribution if and only if  it is the joint distribution of a two-variate random vector $(X,Y)$ such that
$$X=\mathbb{E}(Z|X) \ \ \ \ \text{and} \ \ \ \ Y=\mathbb{E}(Z|Y) \ \ \ \  \text{almost surely}$$
for some random variable $Z$ with $0\le Z\le 1$. \end{prop}

\begin{proof} This is straightforward. See \cite{pitman, mastersthesis}. \end{proof}

Recall the definition of the class $\mathcal{R}$ formulated in the previous section. Let us study the connection between this class and the family of all coherent distributions.

%
%
%

\begin{proof}[Proof of Proposition \ref{m=mu+nu}]  First,  we show that the decomposition $m=\mu+\nu$ exists for all $m\in \mathcal{C}$. Indeed, by virtue of Proposition \ref{def2}, we can find a
random vector $(X,Y)\sim m$ defined on some probability space $(\Omega, \mathcal{F}, \mathbb{P})$, such that $X=\mathbb{E}(Z|X)$ and $Y=\mathbb{E}(Z|Y)$ for some random variable  $Z\in[0,1]$. For a set $C\in \mathcal{B}([0,1]^2)$, we put 
\begin{equation} \label{mu-nu} \mu(C)=\int_{\{(X,Y)\in C\}}Z \ \mathrm{d}\mathbb{P} \ \ \ \ \ \text{and} \ \ \ \ \ \nu(C)=\int_{\{(X,Y)\in C\}}(1-Z) \ \mathrm{d}\mathbb{P}. \end{equation}
Then the equality $m=\mu+\nu$ is evident. Furthermore, for a fixed $A\in \mathcal{B}([0,1])$, we have
\begin{equation} \label{wwo1} \int_{\{X\in A\}}X \ \mathrm{d}\mathbb{P} \ = \ \int_{\{X\in A\}}Z \ \mathrm{d}\mathbb{P} \ = \ \int_{A}1 \ \mathrm{d}\mu^x,\end{equation}
where the first equality is due to $X=\mathbb{E}(Z|X)$ and the second is a consequence of (\ref{mu-nu}).  Moreover, we may also write
\begin{equation} \label{wwo2} \int_{\{X\in A\}} X \ \mathrm{d}\mathbb{P} \ = \ \int_{A\times[0,1]}x \ \mathrm{d}m \ = \
  \int_{A}x \ \mathrm{d}\mu^x  +  \int_{A}x \ \mathrm{d}\nu^x.\end{equation}
Combining (\ref{wwo1}) and (\ref{wwo2}), we get
$$\ \int_{A}(1-x) \ \mathrm{d}\mu^x   \ =   \ \int_{A} x \ \mathrm{d}\nu^x ,$$
for all $A\in \mathcal{B}([0,1])$.  The symmetric condition (the second requirement in Definition \ref{R-set}) is shown analogously. This completes the first part of the proof.

Now, pick a probability measure $m$ on $[0,1]^2$  such that $m=\mu+\nu$ for some $(\mu, \nu) \in \mathcal{R}$. We need to show that $m$ is coherent. To this end, consider the probability space $([0,1]^2, \mathcal{B}([0,1]^2),m)$ and the random variables $X,Y : [0,1]^2 \rightarrow [0,1]$ defined by 
$$X(x,y)=x \ \ \ \ \text{and} \ \ \ \ Y(x,y)=y, \ \ \ \   x,\,y\in [0,1].$$
Additionally, let $Z$ denote the Radon--Nikodym derivative of $\mu$ with respect to $m$: we have $0\le Z \le 1$ $m$--almost surely and
$\mu(C)   =   \int_{C}Z  \mathrm{d}m$ for all $C\in \mathcal{B}([0,1]^2)$. Again by Proposition \ref{def2}, it is sufficient to verify
that $X=\mathbb{E}(Z|X)$ and $Y=\mathbb{E}(Z|Y)$. By symmetry, it is enough to show the first equality.
Fix $A\in \mathcal{B}([0,1])$ and note that
\begin{equation} \label{wwo3} \int_{\{X\in A\}}X \ \mathrm{d}m  \ = \ \int_{A\times[0,1]} x \ \mathrm{d}m \ = \  
\int_{A}x \ \mathrm{d}\mu^x  +  \int_{A}x \ \mathrm{d}\nu^x. \end{equation}
Similarly, we also have
\begin{equation} \label{wwo4}   \int_{\{X\in A\}}Z \ \mathrm{d}m  \ = \ \int_{A\times[0,1]} Z \ \mathrm{d}m \ = \  \mu(A\times[0,1]) \ = \
\int_{A}1 \ \mathrm{d}\mu^x. \end{equation}
Finally,  note that by $(\mu, \nu) \in \mathcal{R}$, the right-hand sides of (\ref{wwo3}) and (\ref{wwo2}) are equal. Therefore we obtain the identity
$$\int_{\{X\in A\}}X \  \mathrm{d}m  \ = \  \int_{\{X\in A\}}Z  \ \mathrm{d}m$$
for arbitrary $A\in \mathcal{B}([0,1])$. This yields  the claim. \end{proof}

We turn our attention to the characterisation of $\mathrm{ext}(\mathcal{C})$ stated in the previous section.

\begin{proof}[Proof of Theorem \ref{char}, the implication `$\Rightarrow$'] Let $m$ be an  extremal coherent measure and suppose, on contrary, that  $(\mu_1, \nu_1)$ and $(\mu_2, \nu_2)$ are two different elements of $\mathcal{R}(m)$. We will prove that $m-\mu_1+\mu_2$ and $m-\mu_2+\mu_1$ are also coherent  distributions. Because of
$$m \ = \ \frac{1}{2}(m-\mu_1+\mu_2) \ +  \ \frac{1}{2}(m-\mu_2+\mu_1),$$
we will obtain the contradiction with the assumed extremality of $m$. By symmetry, it is enough to show that  $(m-\mu_1+\mu_2) \in \mathcal{C}$. 
To this end, by virtue of Proposition \ref{m=mu+nu}, it suffices  to check that $m-\mu_1+\mu_2$ is a probability measure and $(\mu_2, m-\mu_1) \in \mathcal{R}$. First, note that $\nu_1=m-\mu_1$ is nonnegative and fix an arbitrary $A\in \mathcal{B}([0,1])$. As $(\mu_1, \nu_1)$ and $(\mu_2, \nu_2)$ are representations of $m$, Definition  \ref{R-set} gives
$$ \int_{A} 1 \ \mathrm{d}\mu_1^x  \ = \ \int_{A} x \ (\mathrm{d}\nu_1^x+\mathrm{d}\mu_1^x) \ = \  
 \int_{A} x \ \mathrm{d}m^x, $$
 and
\begin{equation} \label{eq1-T1.4} \int_{A} 1 \ \mathrm{d}\mu_2^x  \ = \ \int_{A} x \ (\mathrm{d}\nu_2^x+\mathrm{d}\mu_2^x) \ = \  
 \int_{A} x \ \mathrm{d}m^x, \end{equation}
 so $\mu_1^x(A)=\mu_2^x(A)$. Similarly, we can deduce that $\mu_1^y=\mu_2^y$, which means that marginal distributions of $\mu_1$ and $\mu_2$ are  equal. This, together with  $m-\mu_1\ge 0$, proves that $m-\mu_1+\mu_2$ is a probability measure. Next, using (\ref{eq1-T1.4}) and $\mu_1^x=\mu_2^x$, we can also write
 \begin{equation} \label{eq2-T1.4}  \int_{A} (1-x) \ \mathrm{d}\mu_2^x \ = \ \int_{A} x \ \mathrm{d}m^x - \int_{A} x \ \mathrm{d}\mu_1^x \ = \  \int_{A} x \ \mathrm{d}(m-\mu_1)^x. \end{equation}
In the same way we get
 \begin{equation} \label{eq3-T1.4}  \int_{B} (1-y) \ \mathrm{d}\mu_2^y \ = \  \int_{B} y \ \mathrm{d}(m-\mu_1)^y, \end{equation}
 for all $B\in \mathcal{B}([0,1])$. By (\ref{eq2-T1.4}) and (\ref{eq3-T1.4}), we obtain that $(\mu_2, m-\mu_1) \in \mathcal{R}$ and this completes the proof of the uniqueness.
 
To show the minimality, let $m$ be an extremal coherent measure with the representation $(\mu, \nu)$ (which is unique, as we have just proved).  Consider any nonzero $(\tilde{\mu}, \tilde{\nu})\in \mathcal{R}$ with $\tilde{\mu}\le \mu$ and $\tilde{\nu}\le \nu$.  Then, by the very definition of $\mathcal{R}$,  we have
 $(\mu-\tilde{\mu}, \nu-\tilde{\nu})\in \mathcal{R}$. Therefore, by  Proposition \ref{m=mu+nu}, we get
 $$\alpha^{-1}(\tilde{\mu}+\tilde{\nu}), \ (1-\alpha)^{-1}(m-\tilde{\mu}-\tilde{\nu}) \ \in  \mathcal{C},$$
 where $\alpha=(\tilde{\mu}+\tilde{\nu})([0,1]^2) \in (0,1]$. We have the identity
 \begin{equation} \label{eq4-T1.4} m \ = \ \alpha\cdot \Big( \alpha^{-1}(\tilde{\mu}+\tilde{\nu}) \Big)  +  (1-\alpha)\cdot \Big((1-\alpha)^{-1}(m-\tilde{\mu}-\tilde{\nu}) \Big),\end{equation} 
which combined with the extremality of $m$ yields $m=\alpha^{-1}(\tilde{\mu}+\tilde{\nu})=\alpha^{-1}\tilde{\mu}+\alpha^{-1}\tilde{\nu}$. But $(\alpha^{-1}\tilde{\mu},\alpha^{-1}\tilde{\nu})$ belongs to $\mathcal{R}$, since $(\tilde{\mu},\tilde{\nu})$ does, and hence $(\alpha^{-1}\tilde{\mu},\alpha^{-1}\tilde{\nu})$ is a representation of $m$. By the uniqueness,  we deduce that $(\tilde{\mu}, \tilde{\nu}) = \alpha\cdot(\mu,\nu)$.
\end{proof}

\begin{proof}[Proof of Theorem \ref{char}, the implication `$\Leftarrow$'] Let $m$ be a coherent distribution with the uni\-que and minimal representation $(\mu, \nu)$. To show that  $m$ is extremal, consider the decomposition $m=\beta \cdot m_1+(1-\beta)\cdot m_2$ for some $m_1, m_2 \in \mathcal{C}$ and $\beta \in (0,1)$.
 Moreover, let $(\mu_1, \nu_1)\in \mathcal{R}(m_1)$ and $(\mu_2, \nu_2) \in \mathcal{R}(m_2)$. By the convexity of $\mathcal{R}$, we have
 \begin{equation} \label{eq5-T1.4} (\mu', \nu'):= \ (\beta\mu_1 + (1-\beta)\mu_2,  \ \beta\nu_1 +(1-\beta)\nu_2) \ \in \mathcal{R}(m) \end{equation}
 and hence, by the uniqueness, we get $(\mu', \nu')=(\mu, \nu)$. Then, directly by (\ref{eq5-T1.4}), we have
 \begin{equation} \label{eq6-T1.4} \beta \mu_1\le \mu  \ \ \ \textrm{and} \ \ \ \beta   \nu_1 \le \nu. \end{equation}
Combining this with the minimality of $(\mu, \nu)$, we get $(\beta  \mu_1,\beta \nu_1)=\alpha (\mu,\nu)$  for some $\alpha\in [0,1]$. Since $m=\mu+\nu$ and $m_1=\mu_1+\nu_1$ are probability measures, this gives $\alpha=\beta$ and hence  $(\mu_1, \nu_1)=(\mu, \nu)$. This implies $m=m_1$ and completes the proof.
\end{proof}

\section{Extreme points with finite support}

  In this section we study the geometric structure of the supports of measures belonging  to $\textrm{ext}_f(\mathcal{C}) = \{\eta \in \textrm{ext}(\mathcal{C}) :  |\mathrm{supp}(\eta)|<\infty\}$. Our key result is presented in Theorem \ref{NO-cycle} -- we prove that the support of an extremal coherent distribution cannot contain any axial cycles (see Definition \ref{ax-cycle}).  Let us emphasize that this property has been originally conjectured in \cite{zhu}. We start with a simple combinatorial observation: it is straightforward to check that certain special `alternating' cycles are forbidden.

\begin{defi} \label{ALT-cylce-DEF} Let $\eta$ be a coherent distribution with a unique representation $(\mu, \nu)$ and let $\big((x_i,y_i)\big)_{i=1}^{2n}$  be 
an axial cycle contained in  $\mathrm{supp}(\eta)$. Then $\big((x_i,y_i)\big)_{i=1}^{2n}$ is an alternating cycle if 
$$(x_{2i+1}, y_{2i+1})\in \mathrm{supp}(\mu)  \ \ \ \  \text{and}  \ \ \ \ (x_{2i}, y_{2i})\in \mathrm{supp}(\nu),$$
for all $i=1,2,\dots, n$ (with the convention $x_{2n+1}=x_1,\,y_{2n+1}=y_1$).    \end{defi}

\begin{figure} [H]
\centering
\includegraphics[width=60mm]{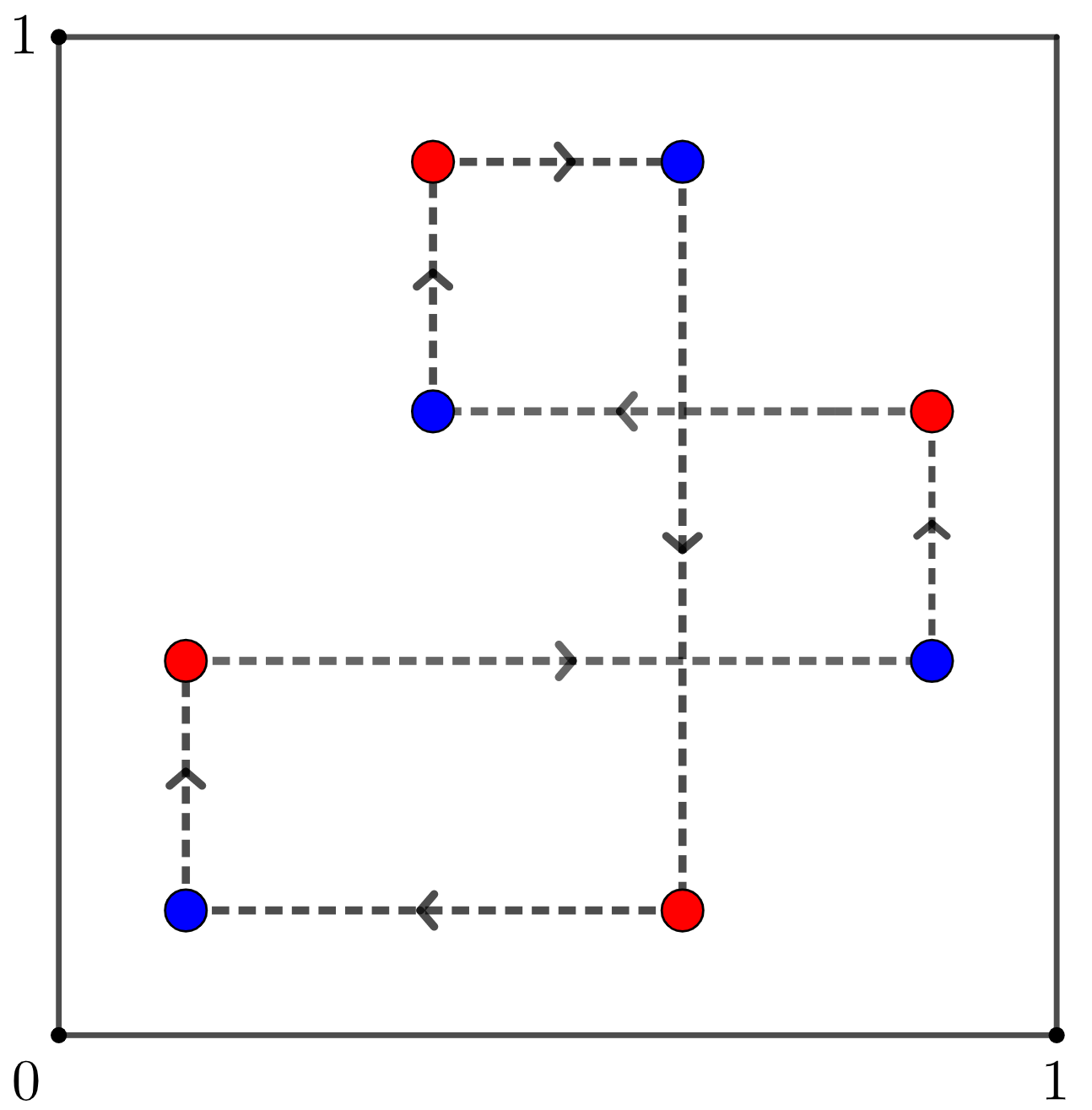}
\vspace{0.2cm}
\caption{An example of an alternating cycle. Red points represent probability masses in $\mathrm{supp}(\mu)$, while blue points indicate probability masses  in $\mathrm{supp}(\nu)$. Arrows outline a possible transformation of  the representation $(\mu, \nu)$.  }
\label{rys1} 
\end{figure}

 \begin{prop}  \label{NO-alt-cycle} If $\eta \in \mathrm{ext}_f(\mathcal{C})$, then $\mathrm{supp}(\eta)$ does not contain any alternating cycles.  \end{prop}

\begin{proof} Let $\eta$ be a coherent distribution with a unique representation $(\mu, \nu)$ and a finite support. Additionally, assume that $\big((x_i,y_i)\big)_{i=1}^{2n}$ is an alternating cycle contained in $\mathrm{supp}(\eta)$. Let $\delta$ be the smaller of the two numbers
$$\min_{0\le i \le n-1}\mu(x_{2i+1}, y_{2i+1})  \ \ \ \ \  \text{and}  \ \ \ \ \ \min_{1\le i \le n}\nu(x_{2i}, y_{2i})$$
(for  brevity, in what follows we will skip the parentheses and write $\mu(a,b)$, $\nu(a,b)$ instead of $\mu(\{a,b\})$, $\nu(\{a,b\})$, respectively). 
By Definition \ref{ALT-cylce-DEF}, we have $\delta>0$.  Now, consider the transformation $(\mu, \nu) \mapsto (\mu', \nu')$ described by the following requirements:

 \begin{enumerate}[label=\arabic*.]
 
 \item for $i=0,1,\dots, n-1$, put
 $$\mu'(x_{2i+1}, y_{2i+1}):= \ \mu(x_{2i+1}, y_{2i+1})-\delta$$
 $$\nu'(x_{2i+1}, y_{2i+1}):= \ \nu(x_{2i+1}, y_{2i+1})+\delta,$$
 
 \item for $i=1,2,\dots, n$, put
 $$\mu'(x_{2i}, y_{2i}):= \ \mu(x_{2i}, y_{2i})+\delta$$
 $$\nu'(x_{2i}, y_{2i}):= \ \nu(x_{2i}, y_{2i})-\delta,$$
  \end{enumerate} 
  
 \begin{enumerate}[label=\arabic*.]
 \setcounter{enumi}{2}
 \item for $(x,y)\not \in \{ (x_i,y_i): 1\le i \le 2n\}$, set
 $$\mu'(x,y)=\mu(x,y),$$
 $$\nu'(x,y)=\nu(x,y).$$
 \end{enumerate} 
Note that $\mu$ and $\mu'$, as well as $\nu$ and $\nu'$, have the same marginal distributions and hence $(\mu', \nu')\in \mathcal{R}$.
We also have $\mu'+\nu'=\mu+\nu=\eta$ and thus $(\mu', \nu')\in \mathcal{R}(\eta)$. This  contradicts the uniqueness of the representation $(\mu, \nu)$ and shows that $\mathrm{supp}(\eta)$ cannot  contain an alternating cycle. By Theorem \ref{char}, this ends the proof.
\end{proof}

 Before the further combinatorial analysis, we need to introduce some useful auxiliary notation.  For $\mu, \nu \in \mathcal{M}([0,1]^2)$ with $|\mathrm{supp}(\mu+\nu)|<\infty$, we define a quotient function $q_{(\mu, \nu)}: \mathrm{supp}(\mu+\nu) \rightarrow [0,1]$ by  
$$q_{(\mu, \nu)}(x,y)= \frac{\mu(x,y)}{\mu(x,y)+\nu(x,y)}.$$
In what follows, we will omit the
subscripts and write $q$ for $q_{(\mu, \nu)}$ whenever the choice for $(\mu, \nu)$ is clear from the context.

\begin{prop}\label{3.1+3.2} Let $\mu, \nu \in \mathcal{M}([0,1]^2)$ and  $|\mathrm{supp}(\mu+\nu)|<\infty$. Then $(\mu, \nu) \in \mathcal{R}$ if and only if 
the following conditions hold simultaneously:

\begin{itemize}[label=\raisebox{0.25ex}{\tiny$\bullet$}]
\item for every $x$ satisfying $\mu(\{x\}\times[0,1])+\nu(\{x\}\times [0,1])>0$, we have
\begin{equation}\label{SUMx} \sum_{\substack{y\in [0,1], \\ (x,y)\in \mathrm{supp}(\mu+\nu)}} q(x,y) \frac{\mu(x,y)+\nu(x,y)}{\mu(\{x\}\times[0,1])+\nu(\{x\}\times [0,1])} \ = \ x,\end{equation}
\item for every $y$ satisfying $\mu([0,1]\times \{y\})+\nu([0,1]\times \{y\})>0$, we have
\begin{equation}\label{SUMy} \sum_{\substack{x\in [0,1], \\(x,y)\in \mathrm{supp}(\mu+\nu)} }  q(x,y) \frac{\mu(x,y)+\nu(x,y)}{\mu([0,1]\times \{y\})+\nu([0,1]\times \{y\})}  \ = \ y,\end{equation}
\end{itemize}
where  sums in (\ref{SUMx}) and (\ref{SUMy}) are well defined -- in both cases, there is only a finite number of nonzero summands.
\end{prop}

\begin{proof} Due to $|\mathrm{supp}(\mu+\nu)|<\infty$, this is a simple consequence of Definition \ref{R-set}.   \end{proof}

Next, we will require an additional distinction between three different types of points.
 
 \begin{defi}   Let $(\mu, \nu)\in \mathcal{R}$. A point $(x,y)\in \mathrm{supp}(\mu+\nu)$ is said to be 
 \begin{enumerate} [label=\alph*)]
 \item a lower out point, if \  $q(x,y)< \min(x,y)$;
  \item an upper out point, if  \ $q(x,y)> \max(x,y)$;
  \item a cut point, if it is not an out point, i.e.
  $$x\le q(x,y) \le y \ \ \ \ \ \text{or} \ \ \ \ \ y\le q(x,y) \le x.$$
 \end{enumerate}
 \end{defi}
 
 Finally, for the sake of completeness, we include a formal definition of an axial path. 
 
\begin{defi}  The sequence $\big((x_i,y_i)\big)_{i=1}^{n}$ with terms in $[0,1]^2$ is called an axial path if
 \begin{enumerate} [label=$\cdot$]
 \item all points $(x_i, y_i)$ are distinct;
  \item we have $x_{i+1} = x_{i}$ or $y_{i+1} = y_{i}$ for all $i$;
  \item there are at most two points on any horizontal or vertical line.
 \end{enumerate}
 \end{defi}

 To develop some intuition, it is convenient to inspect the example given below.

\begin{figure}[H]
\centering
\includegraphics[width=63mm]{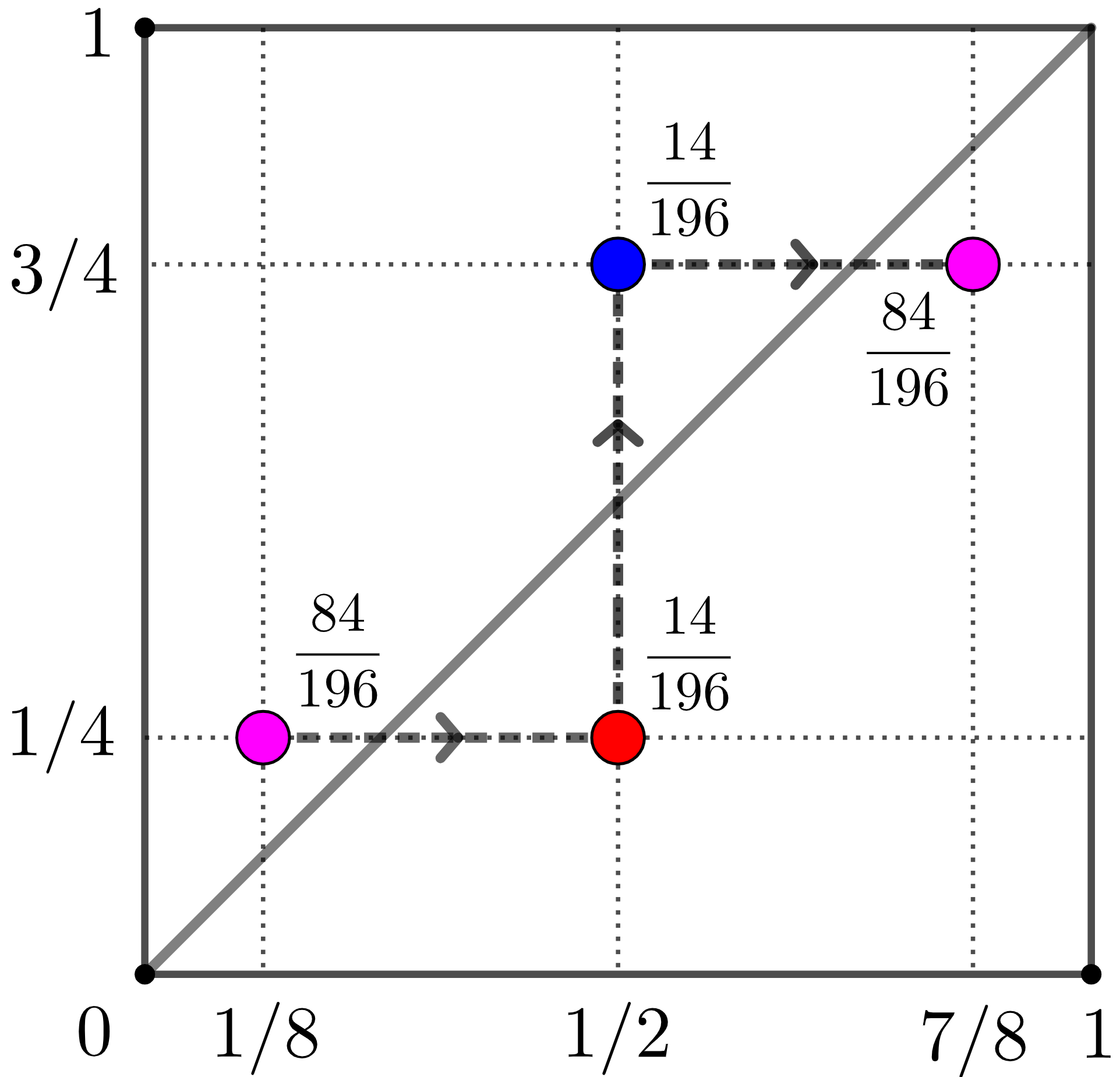}
\vspace{0.15cm}
\caption{ Support of a coherent distribution $m$. Purple points (endpoints of the path) are cut points.
Red point represents a mass in $\mathrm{supp}(\mu)$ and is an upper out point. Blue point indicates a mass  in $\mathrm{supp}(\nu)$ and it is a lower out point.}
\label{rys3} 
\end{figure}
  
\begin{exe}\label{exe1} 
Let $m$ be a probability measure given by
 $$m\Big(\frac{1}{8}, \frac{1}{4}\Big)=\frac{84}{196}, \ \ \ m\Big(\frac{1}{2}, \frac{1}{4}\Big)=\frac{14}{196}, \ \ \  m\Big(\frac{1}{2}, \frac{3}{4}\Big)=\frac{14}{196}, \ \ \  m\Big(\frac{7}{8}, \frac{3}{4}\Big)=\frac{84}{196}.$$
There are five observations, which will be discussed separately.

\smallskip

(i) Consider the decomposition $m=\mu+\nu$, where $(\mu, \nu)$  is determined by the quotient function
$$q\Big(\frac{1}{8}, \frac{1}{4}\Big)=\frac{1}{8}, \ \ \ q\Big(\frac{1}{2}, \frac{1}{4}\Big)=1, \ \ \  q\Big(\frac{1}{2}, \frac{3}{4}\Big)=0, \ \ \  q\Big(\frac{7}{8}, \frac{3}{4}\Big)=\frac{7}{8}.$$
Using Proposition \ref{3.1+3.2}, we can check that $(\mu, \nu)\in \mathcal{R}$. For instance, for $y=\frac{1}{4}$ we get
\begin{equation} \label{3.2-exe} \frac{q(\frac{1}{8}, \frac{1}{4})\cdot m(\frac{1}{8}, \frac{1}{4}) + q(\frac{1}{2}, \frac{1}{4})\cdot m(\frac{1}{2}, \frac{1}{4})}{m(\frac{1}{8}, \frac{1}{4})+m(\frac{1}{2}, \frac{1}{4})} \ = \ \frac{\frac{1}{8}\cdot \frac{84}{196}+1\cdot \frac{14}{196}}{\frac{84}{196}+\frac{14}{196}} \ = \ \frac{1}{4},\end{equation}
which agrees with (\ref{SUMy}). As a direct consequence, by Proposition \ref{m=mu+nu}, we have $m\in \mathcal{C}$. 

\smallskip

(ii) Observe that $(\frac{1}{8}, \frac{1}{4})$ and $(\frac{7}{8}, \frac{3}{4})$ are cut points, $(\frac{1}{2}, \frac{1}{4})$ is an upper out point and $(\frac{1}{2}, \frac{3}{4})$ is a lower out point.  Moreover, $\mathrm{supp}(m)$ is an axial path without cycles -- see Figure \ref{rys3}.

\smallskip

(iii) Notably, $(\mu, \nu)$ is a unique representation of $m$. Indeed, $(\frac{1}{8}, \frac{1}{4})$ is the only point in $\mathrm{supp}(m)$ with $x$-coordinate equal to $\frac{1}{8}$ and hence $q(\frac{1}{8}, \frac{1}{4})=\frac{1}{8}$.  Accordingly, $q(\frac{1}{2}, \frac{1}{4})=1$ is now a consequence of (\ref{3.2-exe}). The derivation of $q(\frac{1}{2}, \frac{3}{4})=0$ and $q(\frac{7}{8}, \frac{3}{4})=\frac{7}{8}$ follows from an analogous computation.

\smallskip

(iv) Finally, the representation $(\mu, \nu)$ is minimal; let $(\tilde{\mu}, \tilde{\nu})\in \mathcal{R}$ satisfy $\tilde{\mu}\le \mu$ and $\tilde{\nu}\le \nu$.
Suppose that $(\frac{1}{8}, \frac{1}{4})\in \mathrm{supp}(\tilde{\mu}+\tilde{\nu})$.  Again, as   $(\frac{1}{8}, \frac{1}{4})$ is the only point in $\mathrm{supp}(m)$ with $x$-coordinate equal to $\frac{1}{8}$, we get $q_{(\tilde{\mu},\tilde{\nu})}(\frac{1}{8}, \frac{1}{4})=\frac{1}{8}$. Next, assume that $(\frac{1}{2}, \frac{1}{4})\in \mathrm{supp}(\tilde{\mu}+\tilde{\nu})$. As 
$\tilde{\nu} (\frac{1}{2}, \frac{1}{4}) \le \nu(\frac{1}{2}, \frac{1}{4})=0$, we have $q_{(\tilde{\mu},\tilde{\nu})}(\frac{1}{2}, \frac{1}{4})=1$.
Likewise, we can check that 
\begin{equation} \label{EQ-exe2} q_{(\tilde{\mu},\tilde{\nu})}(x,y) \ = \ q_{(\mu, \nu)}(x,y) \ \ \ \ \textrm{for all } (x,y)\in \mathrm{supp}(\tilde{\mu}+\tilde{\nu}).\end{equation}
By Proposition \ref{3.1+3.2} and the equation (\ref{EQ-exe2}), we easily obtain that $\tilde{\mu} +\tilde{\nu}=0$ or $\mathrm{supp}(\tilde{\mu}+\tilde{\nu})=\mathrm{supp}(m)$. For example,

\smallskip

$\cdot$  if $(\frac{1}{2}, \frac{1}{4})\in \mathrm{supp}(\tilde{\mu}+\tilde{\nu})$, then (\ref{SUMx}) gives $(\frac{1}{2}, \frac{3}{4}) \in \mathrm{supp}(\tilde{\mu}+\tilde{\nu})$;
 
 \smallskip
 
$\cdot$ if $(\frac{1}{2}, \frac{3}{4})\in \mathrm{supp}(\tilde{\mu}+\tilde{\nu})$, then (\ref{SUMy}) yields $(\frac{7}{8}, \frac{3}{4}) \in \mathrm{supp}(\tilde{\mu}+\tilde{\nu})$.
 
 \smallskip
 
 \noindent Therefore, if $\tilde{\mu}+\tilde{\nu}\not=0$, then the measure $\tilde{\mu}+\tilde{\nu}$ is supported on the same set as $m$ and $q_{(\tilde{\mu},\tilde{\nu})} \equiv  q_{(\mu, \nu)}$. For the same reason, i.e. using Proposition \ref{3.1+3.2} and path structure of $\mathrm{supp}(m)$, it follows that $\tilde{\mu}+ \tilde{\nu} = \alpha \cdot m$ for some $\alpha \in [0,1]$. For instance, by (\ref{SUMy}) for $y=\frac{1}{4}$, we get
 $$\frac{\frac{1}{8}\cdot \tilde{m}(\frac{1}{8}, \frac{1}{4}) + 1\cdot \tilde{m}(\frac{1}{2}, \frac{1}{4})}{\tilde{m}(\frac{1}{8}, \frac{1}{4})+\tilde{m}(\frac{1}{2}, \frac{1}{4})} \ = \ \frac{1}{4},$$
where $\tilde{m}= \tilde{\mu}+\tilde{\nu}.$ Hence $\tilde{m}(\frac{1}{8}, \frac{1}{4})\tilde{m}(\frac{1}{2}, \frac{1}{4})^{-1}=m(\frac{1}{8}, \frac{1}{4})m(\frac{1}{2}, \frac{1}{4})^{-1}=\frac{84}{14}$.

\smallskip

(v) By the above analysis and Theorem \ref{char}, we conclude that $m\in \mathrm{ext}_f(\mathcal{C})$.
\end{exe}
  
 We are now ready to demonstrate the central result of this section. 
 
 \begin{thm} \label{NO-cycle}  If $\eta \in \mathrm{ext}_f(\mathcal{C})$, then $\mathrm{supp}(\eta)$ is an axial path without cycles. \end{thm}
 
 Let us briefly explain the main idea of the proof. For $\eta \in \mathrm{ext}_f(\mathcal{C})$,  we inductively construct a special axial path contained in $\mathrm{supp}(\eta)$, which does not contain any cut points  (apart from the endpoints). We show that axial path obtained in this process is acyclic  and involves all points from $\mathrm{supp}(\eta)$.
 
 \begin{proof}[Proof of Theorem \ref{NO-cycle}] Fix $\eta \in \mathrm{ext}_f(\mathcal{C})$ and let $(\mu, \nu)$ be the unique representation of $\eta$.
 By $\mathcal{L}(\eta)$ and $\mathcal{U}(\eta)$ denote the sets of lower and upper out points, correspondingly.
Choose any $(x_0,y_0)\in \mathrm{supp}(\eta)$. We will consider two separate cases now:

 \medskip

\textbf{Case I:} $(x_0, y_0)$ is an out point. With no loss of generality, we can assume that  $(x_0, y_0) \in \mathcal{L}(\eta)$. 
 We then use the following inductive procedure. 
 \\ \\
\indent 1$^\circ$ Suppose we have successfully found $(x_n, y_n)\in \mathcal{L}(\eta)$ and it is the first time we have chosen a point with the $x$-coordinate equal to $x_n$. Since $(x_n, y_n)\in \mathcal{L}(\eta)$, we have $q(x_n, y_n)<x_n$. By (\ref{SUMx}), there must exist a point $(x_{n+1},y_{n+1})\in \mathrm{supp}(\eta)$ such that $x_{n+1}=x_n$ and $q(x_{n+1}, y_{n+1})>x_n$. We pick one such point and add it at the end of the path. If $(x_{n+1}, y_{n+1})$ is a cut point or an axial cycle was just created, we exit the loop. Otherwise, note that $(x_{n+1}, y_{n+1})\in \mathcal{U}(\eta)$. Go to 2$^\circ$.
\\ \\
\indent 2$^\circ$ Assume we have successfully found $(x_n, y_n)\in \mathcal{U}(\eta)$ and it is the first time we have chosen a point with the $y$-coordinate equal to $y_n$. Since $(x_n, y_n)\in \mathcal{U}(\eta)$, we have $q(x_n, y_n)>y_n$. By (\ref{SUMy}), there must exist a point $(x_{n+1},y_{n+1})\in \mathrm{supp}(\eta)$ such that $y_{n+1}=y_n$ and $q(x_{n+1}, y_{n+1})<y_n$. We pick one such point and add it at the end of the path. If $(x_{n+1}, y_{n+1})$ is a cut point or an axial cycle was just created, we exit the loop. Otherwise, note that $(x_{n+1}, y_{n+1})\in \mathcal{L}(\eta)$. Go to 1$^\circ$.

\begin{figure}[H]
\centering
\includegraphics[width=60mm]{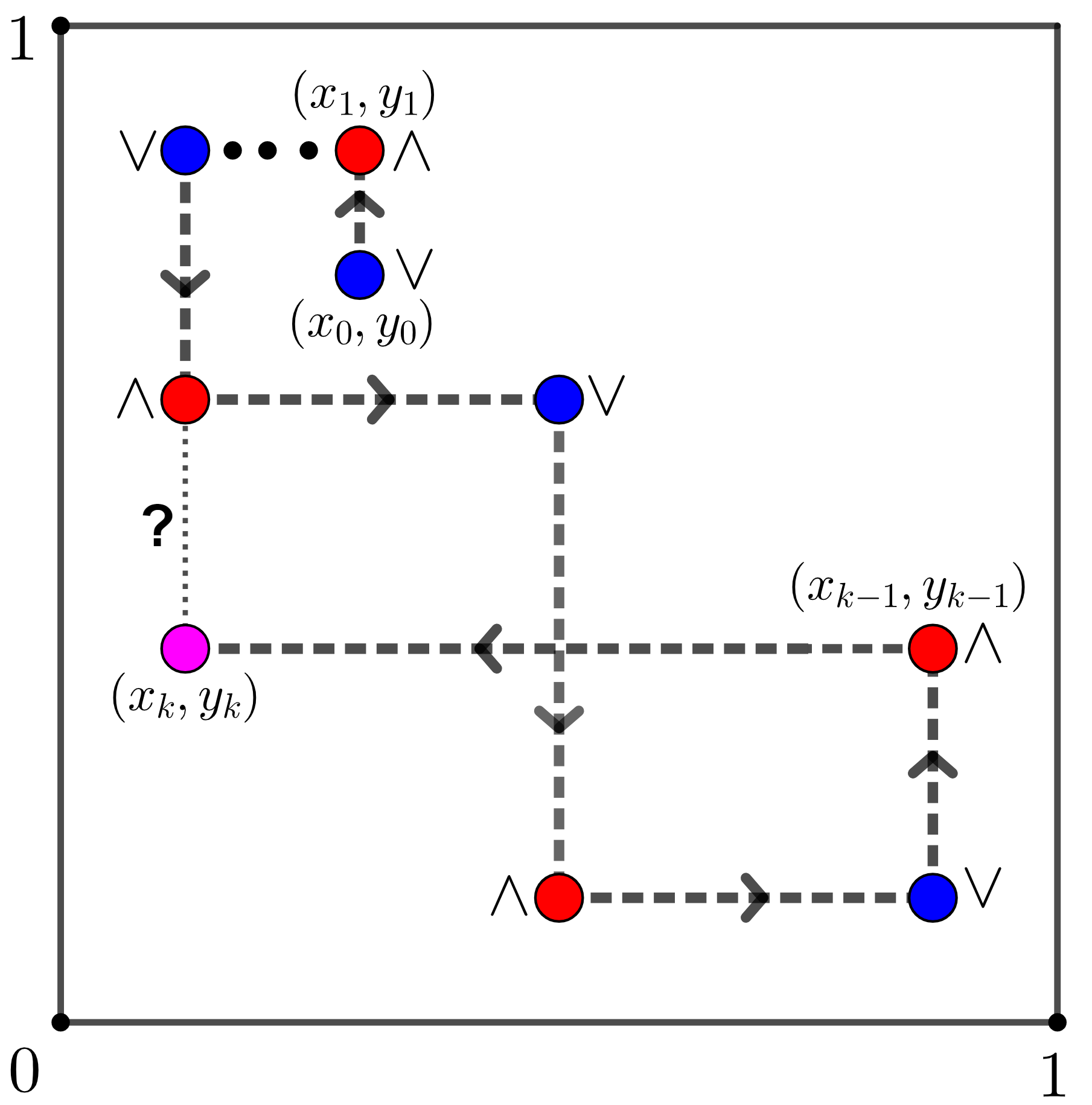}
\vspace{0.2cm}

\caption{An example of an axial path constructed by the algorithm.
Symbols $\vee, \wedge$ are placed next to lower ($\vee$) and upper ($\wedge$) out points.  Purple point $(x_k, y_k)$ is the endpoint of the path.
 Red points represent probability masses in $\mathrm{supp}(\mu)$, while blue points indicate probability masses  in $\mathrm{supp}(\nu)$.\label{rys2}}
 \vspace{0.3cm}
\includegraphics[width=60mm]{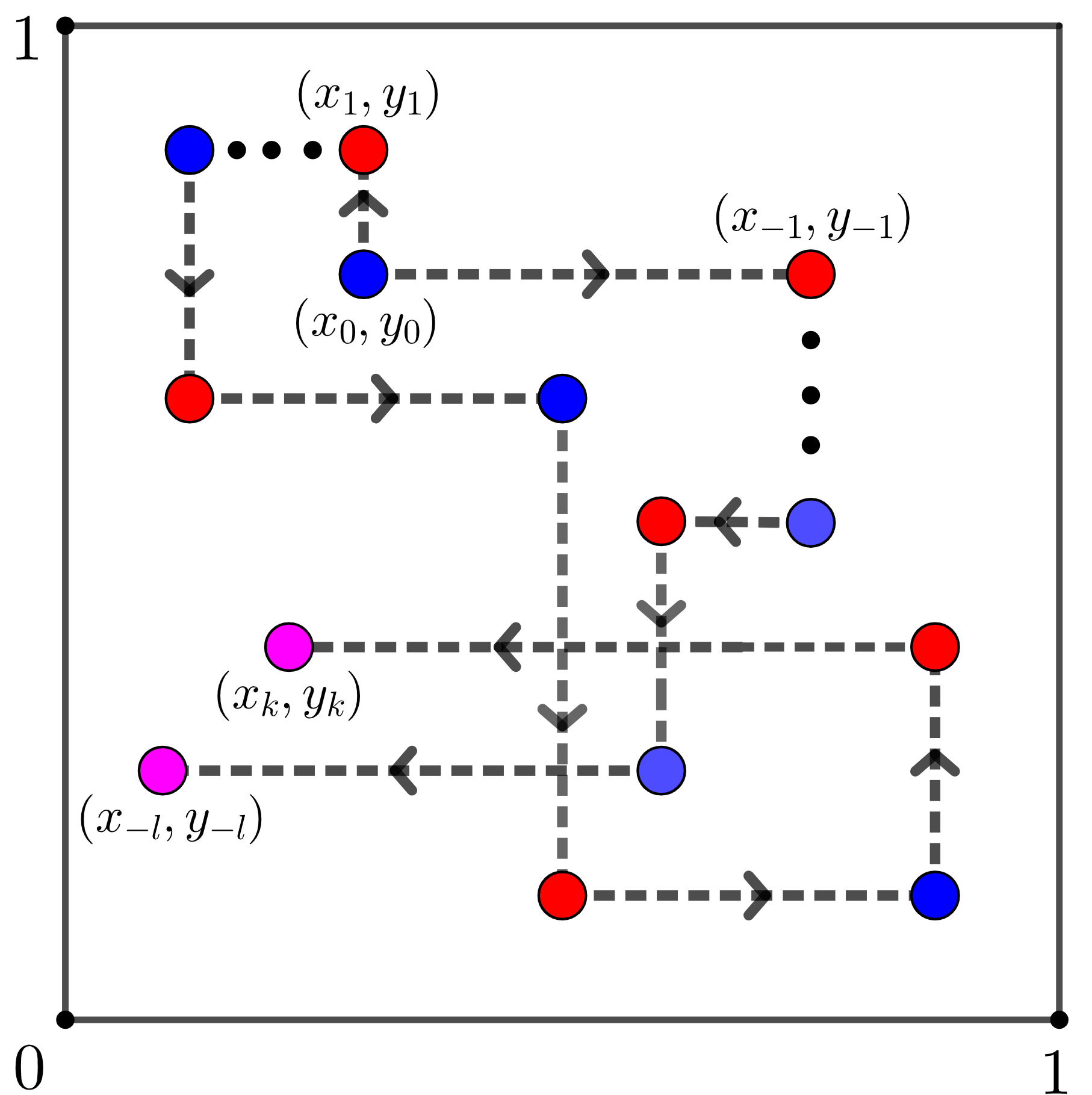}
\vspace{0.2cm}

\caption{An example of an axial path $\Gamma$ constructed after the second run of the algorithm. Purple points $(x_k, y_k)$ and $(x_{-l}, y_{-l})$  (endpoints of $\Gamma$) are cut points. Red points represent probability masses in $\mathrm{supp}(\mu)$, while blue points indicate probability masses  in $\mathrm{supp}(\nu)$.\label{rys4}}
\end{figure}

As $|\mathrm{supp}(\eta)|<\infty$, the procedure terminates after a finite number  of steps (denote it by $k$) and produces an axial path $\big((x_i,y_i)\big)_{i=0}^{k}$ contained in $\mathrm{supp}(\eta)$; to be more precise, it is also formally possible that $(x_k,y_k)$ is a third point on some horizontal or vertical line (in such a case we have obtained an axial cycle).  By the construction of the loop, point $(x_k, y_k)$ is either an endpoint of an axial cycle or a cut point. Let us show that the first alternative is impossible. 
 First, we clearly have $\mathcal{L}(\eta)\subset \mathrm{supp}(\nu)$ and  $\mathcal{U}(\eta)\subset \mathrm{supp}(\mu)$, see Figure \ref{rys2}.  Next, assume that $(x_{k-1}, y_{k-1})\in \mathcal{U}(\eta)$. This means that $(x_k, y_k)$ was found in step 2$^\circ$ and $q(x_k, y_k)<y_{k-1}\le 1$. Therefore $(x_k, y_k)\in \mathrm{supp}(\nu)$ and there exists an alternating cycle in $\mathrm{supp}(\eta)$. However, this is not possible because of Proposition \ref{NO-alt-cycle}. If $(x_{k-1}, y_{k-1})\in \mathcal{L}(\eta)$, the argument is analogous.
\smallskip

We have shown that $(x_{k}, y_k)$ is a cut point. Set $\Gamma_+=\bigcup_{i=1}^k\{(x_i,y_i)\}$. Moving on, we can return to the starting point $(x_0,y_0)$ and repeat the above construction in the reversed direction. By switching the roles of $x$ and $y$-coordinates in steps 1$^\circ$ and 2$^\circ$, we produce  another axial path  $(x_i,y_i)_{i=0}^{-l}$. 
Set $\Gamma_-=\bigcup_{i=-1}^{-l}\{(x_i,y_i)\}$ and
$$\Gamma \ = \ \Gamma_+ \cup \{(x_0,y_0)\} \cup \Gamma_-.$$
Repeating the same arguments as before, we show that $(x_{-l}, y_{-l})$ is a cut point and $\Gamma$ is an axial path  without cycles, see Figure \ref{rys4}.
\smallskip

It remains to verify that $\mathrm{supp}(\eta)=\Gamma$. This will be accomplished by showing that there exists $(\tilde{\mu}, \tilde{\nu})\in \mathcal{R}$ with $\tilde{\mu}\le \mu$, $\tilde{\nu}\le \nu$ and $\mathrm{supp}(\tilde{\mu}+\tilde{\nu})=\Gamma$. This will give the claim: by the minimality of the representation $(\mu, \nu)$, we will deduce that $\tilde{\mu}+ \tilde{\nu} = \alpha \cdot \eta$ for some $\alpha \in (0,1]$, and hence $\mathrm{supp}(\tilde{\mu}+\tilde{\nu})=\mathrm{supp}(\eta)$. 
\smallskip

We begin with the endpoints of $\Gamma$. As $(x_k, y_k)$ is a cut point, there exists $\gamma \in [0,1]$ such that $q(x_k, y_k)=\gamma x_k + (1-\gamma)y_k$. We can write 
\begin{equation} \label{cut1} \eta(x_k, y_k) \ =  \ \eta'(x_k, y_k) + \eta''(x_k, y_k),\end{equation}
where $\eta'(x_k, y_k)=\gamma\eta(x_k, y_k)$ and  $\eta''(x_k, y_k)=(1-\gamma)\eta(x_k, y_k)$. Set
\begin{equation} \label{cut2} \mu'(x_k, y_k)  =  x_k\eta'(x_k, y_k) \ \ \ \ \text{and} \ \ \ \  \mu''(x_k, y_k)=y_k\eta''(x_k, y_k).\end{equation}
By (\ref{cut1}) and (\ref{cut2}), we have
\begin{equation} \label{cut3} \mu'(x_k, y_k)+\mu''(x_k, y_k)  =  \Big(x_k\gamma+y_k(1-\gamma)\Big)\eta(x_k,y_k)  =  \mu(x_k, y_k).\end{equation}
Equations (\ref{cut1}) and (\ref{cut3}) have a clear and convenient interpretation. Namely, we can visualize it as `cutting' the point $(x_k, y_k)$ into
two separate points: $(x_k, y_k)'$ with mass $\eta'(x_k, y_k)$ and $(x_k, y_k)''$ with mass $\eta''(x_k, y_k)$. Moreover,  calculating their
quotient functions independently, we get $q'(x_k, y_k)=x_k$ and $q''(x_k, y_k)=y_k$. Performing the same `cut' operation on $(x_{-l}, y_{-l})$ we 
can divide this point into $(x_{-l}, y_{-l})'$ and $(x_{-l}, y_{-l})''$ such that $q'(x_{-l}, y_{-l})=x_{-l}$ and $q''(x_{-l}, y_{-l})=y_{-l}$.
\smallskip

Observe that $(x_k, y_k)$ and $(x_{k-1}, y_{k-1})$ have exactly one common coordinate, say $y_k=y_{k-1}$. Consequently, $(x_k, y_k)$ is the only point in $\Gamma$ with $x$-coordinate equal to $x_k$. Additionally, by (\ref{SUMy}) and  $(x_{k-1}, y_{k-1}) \in \mathcal{U}(\eta)$, this means that $q(x_k, y_k)\not=y_{k}$ and $\gamma>0$. Hence $\eta'(x_k, y_k)>0$. Similarly, suppose that $y_{-l}=y_{-l+1}$ (as presented in Figure \ref{rys4};  for other configurations of endpoints, we proceed by analogy). Thus, $(x_{-l}, y_{-l})$ is the only point in $\Gamma$ with $x$-coordinate equal to $x_{-l}$. By (\ref{SUMy}) and $(x_{-l+1}, y_{-l+1}) \in \mathcal{L}(\eta)$, we have $\eta'(x_{-l}, y_{-l})>0$. 

\smallskip

 Next, consider the following function $\tilde{q}:\Gamma \rightarrow [0,1]$ uniquely determined by the following requirements:
\begin{enumerate}[label=\arabic*.]
\item $\tilde{q}(x_k, y_k)=x_k$ (if $y_k=y_{k-1}$, as we have assumed) 
\\ or \  $\tilde{q}(x_k, y_k)=y_k$ (in the case when $x_k=x_{k-1}$),
\smallskip
\item  $\tilde{q}(x_{-l}, y_{-l})=x_{-l}$ (if $y_{-l}=y_{-l+1}$, as we have assumed) 
\\ or \  $\tilde{q}(x_{-l}, y_{-l})=y_{-l}$ (in the case when $x_{-l}=x_{-l+1}$),
\smallskip
\item  $\tilde{q}(x,y)=0$ for all $(x,y)\in \Gamma \cap \mathcal{L}(\eta)$,
\item $\tilde{q}(x,y)=1$ for all $(x,y)\in \Gamma \cap \mathcal{U}(\eta)$.
\end{enumerate}
\smallskip
 Set $\delta=\min(a,b,c,d)$, where 
\begin{center} $a= \eta'(x_k, y_k)$ \ (if $y_k=y_{k-1}$)  \ \ \ or \ \ \ $a=\eta''(x_k, y_k)$ \ (if $x_{k}=x_{k-1}$),\end{center}
\begin{center} $b= \eta'(x_{-l}, y_{-l})$ \ (if $y_{-l}=y_{-l+1}$)  \ \ \ or \ \ \ $b=\eta''(x_{-l}, y_{-l})$ \ (if $x_{-l}=x_{-l+1}$),\end{center}
 \begin{center}$c  =  \min_{(x,y)\in \Gamma \cap \mathcal{L}(\eta)}\nu(x, y)$, \ \ \  $d  =  \min_{(x,y)\in \Gamma \cap \mathcal{U}(\eta)}\mu(x, y).$\end{center}
 Then $\delta>0$, which follows from the previous discussion. Finally, using the acyclic path structure of $\Gamma$ and Proposition \ref{3.1+3.2} (just as in Example \ref{exe1}), we are able to find a pair $(\tilde{\mu}, \tilde{\nu})\in \mathcal{R}$ with $\mathrm{supp}(\tilde{\mu}+\tilde{\nu})=\Gamma$ and a quotient function $q_{(\tilde{\mu}, \tilde{\nu})}=\tilde{q}$. Letting
 $$\beta \ = \ \delta\cdot \Big(\max_{(x,y)\in \Gamma}(\tilde{\mu}+\tilde{\nu})(x,y) \Big)^{-1},$$
 we see that $\beta\tilde{\mu} \le \mu$ and  $\beta\tilde{\nu} \le \nu$, as desired.

\smallskip
 
\textbf{Case II:} $(x_0, y_0)$ is a cut point. Suppose that $x_0=y_0$ and $q(x_0, x_0)=x_0$. Put
 $$\tilde{\mu}  =  \mathbbm{1}_{\{(x_0, x_0)\}} x_0 \eta(x_0, y_0) \ \ \ \ \text{and} \ \ \ \  \tilde{\nu}=\mathbbm{1}_{\{(x_0, x_0)\}}(1-x_0)\eta(x_0, y_0).$$
 We have $(\tilde{\mu}, \tilde{\nu})\in \mathcal{R}$ and $\tilde{\mu}\le \mu$,  $\tilde{\nu}\le \nu$. Hence $\mathrm{supp}(\eta)=\{(x_0, x_0)\}$.
  Next, assume that  $x_0\not=y_0$. In that case, $q(x_0, y_0)$ cannot be equal to both $x_0$ and $y_0$ at the same time.  
 This means that we can proceed just as in Case I (at least in one direction). The only difference is that we have 
 already located one of the cut points -- there is no need to apply the procedure twice.
 \end{proof}
  
  From the proof provided, we can deduce yet another significant conclusion.

\begin{cor}\label{01-UL}  If $\eta \in \mathrm{ext}_f(\mathcal{C})$, then $q(x,y)=0$  for all  $(x,y)\in\mathcal{L}(\eta)$ 
and
$q(x,y)=1$ for all $(x,y)\in\mathcal{U}(\eta)$. 
Except for the endpoints of this axial path (which are cut points), $\mathrm{supp}(\eta)$ consists of lower and upper out points, appearing alternately. \end{cor}

\begin{proof} Note that  $\mathcal{L}(\eta)$ and $\mathcal{U}(\eta)$ are well defined as the representation of $\eta$ is unique.
The statement follows directly from the proof of Theorem \ref{NO-cycle}.
 \end{proof}

\section{Asymptotic estimate}
 Equipped with the machinery developed in the previous sections, we are ready to establish the asymptotic estimate (\ref{2/e}). We need  to clarify how the properties of $\mathrm{ext}_f(\mathcal{C})$ covered in the preceding part apply to this problem. Referring to the prior notation, we will write
 $$(X, Y) \in \mathcal{C}_f   \ \ \ \ \  \text{or}  \ \ \ \ \ (X, Y) \in   \mathrm{ext}_f(\mathcal{C}),$$
to indicate that the distribution of a random vector $(X,Y)$ is a coherent (or an extremal coherent) measure with finite support.

\begin{prop} For any $\alpha>0$, we have
$$ \sup_{(X,Y)\in \mathcal{C}}\mathbb{E}|X-Y|^{\alpha} \ \ = \ \ \sup_{(X,Y)\in \mathcal{C}_f}\mathbb{E}|X-Y|^{\alpha}.$$
\end{prop}

\begin{proof} Fix any $(X,Y)\in \mathcal{C}$. As shown in \cite{contra, mastersthesis}, there exists a sequence $(X_n,Y_n)\in \mathcal{C}_f$ such that
\begin{equation} \label{XnYn} \max\Big\{|X-X_n|, |Y-Y_n| \Big\} \ \le \ \frac{1}{n}, \ \ \ \ \ \  \mbox{for all}\   n=1,\,2,\,\ldots \end{equation}
almost surely.  Consequently, by dominated convergence and (\ref{XnYn}), we obtain
$$\mathbb{E}|X-Y|^{\alpha} \ \ = \ \ \lim_{n \rightarrow \infty}  \mathbb{E}|X_n-Y_n|^{\alpha},$$
and thus
$$\mathbb{E}|X-Y|^{\alpha} \ \  \le   \ \ \sup_{n\in \mathbb{N}} \  \mathbb{E}|X_n-Y_n|^{\alpha} \ \ \le  \ \ \sup_{(X,Y)\in \mathcal{C}_f}  \mathbb{E}|X-Y|^{\alpha}.$$
This proves the `$\le$'-inequality, while in the reversed direction it is evident.
 \end{proof}

Next, we will apply the celebrated Krein--Milman theorem, see \cite{rudin}.

\begin{thm}[Krein--Milman]  A compact convex subset of a Hausdorff locally convex topological vector space is equal to the closed convex hull of its extreme points. \end{thm}

The above statement enables us to restrict the analysis of the estimate \eqref{2/e} to extremal measures. Precisely, we have the following statement.

\begin{prop} \label{Cf==extf} For any $\alpha>0$, we have
$$ \sup_{(X,Y)\in \mathcal{C}_f}\mathbb{E}|X-Y|^{\alpha} \ \ = \ \ \sup_{(X,Y)\in \mathrm{ext}_f(\mathcal{C})}\mathbb{E}|X-Y|^{\alpha}.$$ \end{prop}

\begin{proof} Let $Z=C([0,1]^2, \mathbb{R})$; then $Z^*$ is the space of finite signed Borel measures with the total variation norm $\| \cdot \|_{\mathrm{TV}}$. Let us equip $Z^*$ with the topology of weak$^*$ convergence. Under this topology, $Z^*$ is a Hausdorff and a locally convex space. For a fixed $m\in \mathcal{C}_f$, let 
$$\mathcal{C}_m \ = \ \{m'\in \mathcal{C}_f: \ \mathrm{supp}(m')\subseteq \mathrm{supp}(m) \}$$
denote the family of  coherent distributions supported on the subsets of $\mathrm{supp}(m)$. Firstly, observe that $\mathcal{C}_m$ is convex. Secondly, we can easily verify that $\mathrm{ext}(\mathcal{C}_m)=\mathcal{C}_m\cap \mathrm{ext}_f(\mathcal{C})$. Plainly, if $m'\in \mathcal{C}_m$ and $m'=\alpha\cdot m_1 + (1-\alpha)\cdot m_2$ for some $\alpha\in (0,1)$ and $m_1, m_2 \in \mathcal{C}$, then $\mathrm{supp}(m')= \mathrm{supp}(m_1) \cup \mathrm{supp}(m_2)$ and we must have $m_1, m_2 \in \mathcal{C}_m$. Hence $\mathrm{ext}(\mathcal{C}_m)\subset \mathrm{ext}_f(\mathcal{C})$, whereas $\mathrm{ext}_f(\mathcal{C})\cap \mathcal{C}_m \subset \mathrm{ext}(\mathcal{C}_m)$ is obvious.

 Moreover, we claim that $\mathcal{C}_m$ is compact in the weak$^*$ topology. Indeed, by the Banach--Alaoglu theorem, 
$$B_{Z^*} \ = \ \{\mu \in Z^*: \ \|\mu\|_{\mathrm{TV}}\le 1\}$$  is weak$^*$ compact. As $\mathcal{C}_m\subset B_{Z^{*}}$, it remains  to check that
$\mathcal{C}_m$ is weak$^*$ closed. We can write $\mathcal{C}_m=\mathcal{C}\cap \mathcal{P}_m$, where $\mathcal{P}_m$ stands for the set of all probability measures supported on the subsets of $\mathrm{supp}(m)$. Note that  $\mathcal{P}_m$  is clearly weak$^*$ closed. Lastly, coherent distributions
on $[0,1]^2$ are also  weak$^*$  closed, as demonstrated in \cite{pitman}.

Thus, by  Krein--Milman theorem, there exists a sequence $(m_n)_{n=1}^{\infty}$ with values in $\mathcal{C}_m$, satisfying 
 \begin{equation} \label{m_n:convexhull} m_n \ = \ \beta_1^{(n)}\eta_1^{(n)}+\beta_2^{(n)}\eta_2^{(n)}+\dots +\beta_{k_n}^{(n)}\eta_{k_n}^{(n)}, \end{equation}
where \ $\eta_{1}^{(n)}, \dots, \eta_{k_n}^{(n)} \in \mathrm{ext}(\mathcal{C}_m)$ \ and $\beta_{1}^{(n)}, \dots, \beta_{k_n}^{(n)}$ are positive numbers summing up to $1$, such that
\begin{equation} \label{m_n->*m} \int_{[0,1]^2}f \ \mathrm{d}m_n \ \longrightarrow \ \int_{[0,1]^2}f \ \mathrm{d}m, \end{equation}
 for all bounded, continuous functions $f:[0,1]^2\rightarrow \mathbb{R}$. Put $f(x,y)=|x-y|^{\alpha}$. By (\ref{m_n->*m})  and (\ref{m_n:convexhull}), we have
 \begin{align*}\int_{[0,1]^2}|x-y|^{\alpha} \ \mathrm{d}m \ \ \ &\le  \ \ \ \sup_{n\in \mathbb{N}}  \int_{[0,1]^2}|x-y|^{\alpha} \ \mathrm{d}m_n\\
 &\le  \ \ \ \sup_{\substack{n\in \mathbb{N}, \\1\le i \le k_n}}  \int_{[0,1]^2}|x-y|^{\alpha} \ \mathrm{d}\eta_i^{(n)}\\
   &\le \ \ \ \sup_{\eta \in  \mathrm{ext}_f(\mathcal{C})}  \int_{[0,1]^2}|x-y|^{\alpha} \ \mathrm{d}\eta,\end{align*}
 and hence
 $$ \sup_{(X,Y)\in \mathcal{C}_f}\mathbb{E}|X-Y|^{\alpha} \ \ \le \ \ \sup_{(X,Y)\in \mathrm{ext}_f(\mathcal{C})}\mathbb{E}|X-Y|^{\alpha}.$$ 
 The reverse inequality is clear.
\end{proof}

Now, we have the following significant reduction. Denote by $\mathcal{S}$ the family of all finite sequences $\textbf{z}=(z_0, z_1, \dots, z_{n+1})$, \ $n\in \mathbb{N}$,  with $z_0=z_{n+1}=0$,  $\sum_{i=1}^n z_i=1$   and   $z_i>0$ for $i=1,2,\dots, n$. We emphasize that $n=n(\textbf{z})$, the length of $\textbf{z}$, is also allowed to vary. In what follows, we will write $n$ instead of $n(\textbf{z})$; this should not lead to any confusion. 

\begin{prop} \label{SUP-->PHI} For any $\alpha \ge 1$, we have
\begin{equation} \label{EqLemma0} \sup_{(X,Y) \in \mathrm{ext}_f(\mathcal{C})}\mathbb{E}|X-Y|^{\alpha}  \ \ =  \ \ \sup_{{\textbf{z}\in \mathcal{S}}}\ \sum_{i=1}^{n} z_i  \Big| \frac{z_i}{z_{i-1}+z_i}-\frac{z_i}{z_i+z_{i+1}} \Big|^{\alpha}.\end{equation}
\end{prop}

\begin{proof} 
Consider an arbitrary $\eta \in \mathrm{ext}_f(\mathcal{C})$ and let $(\mu, \nu)$ be its unique representation. Recall, based on Theorem \ref{NO-cycle},  that $\mathrm{supp}(\eta)$ is an axial path without cycles. Set $\mathrm{supp}(\eta)=\{(x_i, y_i) \}_{i=1}^n$ and let $q: \mathrm{supp}(\eta) \rightarrow [0,1]$ be the quotient function associated with $(\mu, \nu)$. In this setup, by (\ref{SUMx}) and (\ref{SUMy}), we can write
\begin{equation} \label{EqLemma1} \int_{[0,1]^2}|x-y|^{\alpha} \ \mathrm{d}\eta \ \ = \ \ \sum_{i=1}^n z_i \Big|\frac{q_{i-1}z_{i-1}+q_iz_i}{z_{i-1}+z_i} -\frac{q_iz_i+q_{i+1}z_{i+1}}{z_i+z_{i+1}} \Big|^{\alpha}, \end{equation}
where  \ \ $z_0=z_{n+1}=0$, \ \ $q_0=q_{n+1}=0$,
\begin{center}$q_i=q(x_i, y_i)$  \ \ \ and \ \ \ $z_i=\eta(x_i, y_i)$,  \ \ \ \ \ \  for all \  $i=1,\,2,\,\dots,n$.\end{center}
Note that if $n=1$, then both sides of \eqref{EqLemma1} are equal to zero; hence $\eta$ does not bring any contribution to \eqref{EqLemma0}. Hence, from now on, we will assume that $n\geq 2$. Notice that by Corollary \ref{01-UL}, the sequence $(q_1, q_2, \dots, q_n)$ is given by
\begin{center} $(q_1, 0, 1, 0, 1, \dots, q_n)$ \ \ \ or \ \ \ $(q_1, 1, 0, 1, 0, \dots, q_n)$ \end{center}
-- except for $q_1$ and $q_n$,  $(q_2, \dots, q_{n-1})$ is simply an alternating binary sequence. 
Furthermore, the right-hand side of (\ref{EqLemma1}) is the sum of
\begin{equation} \label{EqLemma2}P(q_1):= \  z_1 \Big|q_1 -\frac{q_1z_1+q_{2}z_{2}}{z_1+z_{2}} \Big|^{\alpha}
\  + \   z_2 \Big|\frac{q_{1}z_{1}+q_2z_2}{z_{1}+z_2} -\frac{q_2z_2+q_{3}z_{3}}{z_2+z_{3}} \Big|^{\alpha} \end{equation}
 and some other terms not involving $q_1$.   Since $\alpha \ge1$, $P$ is a convex function on $[0,1]$ and hence it is maximized by some $q_1'\in \{0,1\}$; in the case of $P(0)=P(1)$, we choose $q_1'$ arbitrarily.  Depending on $q_1'$,  we shall now perform one of the following transformations $(q, z) \mapsto (\tilde{q}, \tilde{z})$:

\smallskip

a. If $q_1'\not=q_2$, we let $\tilde{n}=n$, $\tilde{q}_1=q_1'$ and $\tilde{q}_i=q_i$ for $i \in \{0\} \cup \{2,3,\dots, n+1\}$, $
\tilde{z}_i=z_i$ for $i \in \{0,1,\dots, n+1\}$. 
This operation only changes $q_1$ into $q_1'$ -- we increase the right-hand side of (\ref{EqLemma1}) by ``correcting'' the quotient function on the first atom.

\smallskip

b. If $q_1'=q_2$, we take  $\tilde{n}=n-1$, $\tilde{q}_0=0$, $\tilde{z}_0=0$ and
$$\tilde{q}_i=q_{i+1}, \ \ \qquad  \tilde{z}_i=\frac{z_{i+1}}{z_2+z_3+\ldots+z_n}  \ \qquad \  \mbox{ for  }\ \ \   i \in \{1,2,\dots, \tilde{n}+1\}.$$
This modification removes the first atom and rescales the remaining ones. It is easy to see that for the transformed sequences $(\tilde{q},\tilde{z})$, the right-hand side of \eqref{EqLemma1} does not decrease. 
\smallskip

Performing a similar transformation for the last summand in \eqref{EqLemma1} (
depending on $q_n'$ and $q_{n-1}$) we obtain a pair of sequences $(\tilde{q}, \tilde{z})$, such that $(\tilde{q}_1, \dots, \tilde{q}_{\tilde{n}})$ is an alternating binary sequence and
 \begin{align*} \int_{[0,1]^2}|x-y|^{\alpha} \ \mathrm{d}\eta \ \ \ &\le  \ \ \  \sum_{i=1}^{\tilde{n}} \tilde{z}_i \Big|\frac{\tilde{q}_{i-1}\tilde{z}_{i-1}+\tilde{q}_i\tilde{z}_i}{\tilde{z}_{i-1}+\tilde{z}_i} -\frac{\tilde{q}_i\tilde{z}_i+\tilde{q}_{i+1}\tilde{z}_{i+1}}{\tilde{z}_i+\tilde{z}_{i+1}} \Big|^{\alpha}\\
 &=  \ \ \  \sum_{i=1}^{\tilde{n}} \tilde{z}_i  \Big| \frac{\tilde{z}_i}{\tilde{z}_{i-1}+\tilde{z}_i}-\frac{\tilde{z}_i}{\tilde{z}_i+\tilde{z}_{i+1}} \Big|^{\alpha}\\
   &\le \ \ \sup_{\tilde{\textbf{z}}}\ \sum_{i=1}^{n} z_i  \Big| \frac{z_i}{z_{i-1}+z_i}-\frac{z_i}{z_i+z_{i+1}} \Big|^{\alpha},\end{align*}
which proves the  inequality `$\le$' in  (\ref{EqLemma0}). The reverse bound follows by a straightforward construction, involving measures with quotient functions equal to $0$ or $1$ (see \eqref{EqLemma1}).
\end{proof}

We require some further notation.  Given $\alpha>0$, let $\Phi_{\alpha}:\mathcal{S}\rightarrow [0,1]$ be defined by
$$\Phi_{\alpha}(z) \ \ = \ \ \sum_{i=1}^{n} z_i  \Big| \frac{z_i}{z_{i-1}+z_i}-\frac{z_i}{z_i+z_{i+1}} \Big|^{\alpha}.$$
By the  preceding discussion, for $\alpha\ge 1$ we have
$$\sup_{(X,Y)\in \mathcal{C}}\mathbb{E}|X-Y|^{\alpha} \  =  \ \sup_{z\in \mathcal{S}}\Phi_{\alpha}(z),$$
and our main problem amounts to the identification of
\begin{equation} \label{sup-NEW} \limsup_{\alpha \to \infty} \Big[  \alpha \cdot \sup_{z\in \mathcal{S}}\Phi_{\alpha}(z)\Big].\end{equation}
It will later become clear that $\limsup$ in (\ref{sup-NEW}) can be replaced by an ordinary limit. We begin by making some introductory observations.

\begin{defi} Fix  \ $\alpha\ge 1$ \ and let \ $\textbf{z}=(z_0, z_1, \dots, z_{n+1})$ \ be a generic element  of \ $\mathcal{S}$.  For \ $1\le i \le n$, we say that
the term (component) $z_i$ of $\textbf{z}$ is \emph{significant} if
\begin{center}  $\sqrt{\alpha} \cdot z_{i-1} <  z_i$ \ \ \ \ and \ \ \ \ $\sqrt{\alpha} \cdot z_{i}  <   z_{i+1}$,  \end{center}
or
\begin{center} $z_{i-1} >  \sqrt{\alpha} \cdot z_i$ \ \ \ \ and \ \ \ \ $z_{i}  >  \sqrt{\alpha} \cdot  z_{i+1}$. \end{center}
The set of  all significant components of $z$ will be denoted by $\phi_\alpha(z)$. 
Whenever a component  $z_i$   of   $\textbf{z}$ ($1\le i \le n$) is not  significant, we say that  $z_i$  is \emph{negligible}.
The terms $z_0$ and $z_{n+1}$ will be treated as neither significant nor negligible. 
 \end{defi}

Now we will show that the contribution of all negligible terms of $z$ to the total sum $\Phi_{\alpha}(z)$ vanishes in the limit $\alpha\to \infty$. Precisely, we have the following.

\begin{prop} \label{onlySignificant} For $\alpha\ge1$ and $z\in \mathcal{S}$, we have
$$ \Phi_{\alpha}(z) \ \ \le \ \ \Psi_{\alpha}(z) \ + \ \Big|1-\frac{1}{1+\sqrt{\alpha}}\Big|^{\alpha},$$
where $\Psi_{\alpha}:\mathcal{S}\rightarrow [0,1]$ is defined by
$$\Psi_{\alpha}(z) \ \ = \ \ \sum_{z_i \in \phi_{\alpha}(z)} z_i  \Big| \frac{z_i}{z_{i-1}+z_i}-\frac{z_i}{z_i+z_{i+1}} \Big|^{\alpha}.$$
\end{prop}
\begin{proof} Since $z_1+z_2+\dots+z_n=1$, it is sufficient to show that
\begin{equation} \label{negligible-eq}  \Big| \frac{z_i}{z_{i-1}+z_i}-\frac{z_i}{z_i+z_{i+1}} \Big| \ \ \le \ \  \Big|1-\frac{1}{1+\sqrt{\alpha}}\Big|, \end{equation}
for all  negligible  components  $z_i$. Assume that (\ref{negligible-eq}) does not hold. Since the ratios $z_i/(z_{i-1}+z_i)$ and $z_{i+1}/(z_i+z_{i+1})$ take values in $[0,1]$, we must have
\begin{equation} \label{significant-eq1}  \min\Big\{\frac{z_i}{z_{i-1}+z_i}, \frac{z_i}{z_i+z_{i+1}}  \Big\} \ \ < \ \ \frac{1}{1+\sqrt{\alpha}} \end{equation}
and
\begin{equation} \label{significant-eq2}  \max\Big\{\frac{z_i}{z_{i-1}+z_i}, \frac{z_i}{z_i+z_{i+1}}  \Big\} \ \ > \ \ \frac{\sqrt{\alpha}}{1+\sqrt{\alpha}}. \end{equation}
It remains to note that component $z_i$ fulfilling (\ref{significant-eq1}) and (\ref{significant-eq2}) is significant.
 \end{proof}

It is also useful to consider some special arrangements consisting of three successive components  $(z_{i-1}, z_i, z_{i+1})$ of the generic sequence $z\in \mathcal{S}$.

\begin{defi} Let \ $\textbf{z}= (z_0, z_1, \dots, z_{n+1})$ \ be an element of \ $\mathcal{S}$.   For \ $1\le i \le n$, \ we say that a subsequence \ $(z_{i-1}, z_i, z_{i+1})$ \ of \ $\textbf{z}$ is
\begin{itemize}[label=\raisebox{0.25ex}{\tiny$\bullet$}]
\item a \emph{split}, if \ \ $z_{i-1}   >   z_i   < z_{i+1}$,
\item a \emph{peak}, if \ \ $z_{i-1}<z_i>z_{i+1}$.
\end{itemize}\end{defi}

\noindent In what follows, let $\mathcal{S}'$  be the subset of all those $z\in \mathcal{S}$, which satisfy:

\begin{enumerate}[label=\arabic*.]
\item  $z_{i-1}\not= z_i$ \ for all \ $i\in \{1,2,\dots,n+1\}$, \smallskip

\item there are no split subsequences in $z$,\smallskip

\item there is exactly one peak in $z$, \smallskip

\item   there  is exactly one negligible component $z_{j_0}$ in $z$, and $z_{j_0}$ is the center of the unique peak $(z_{j_0-1}, z_{j_0}, z_{j_0+1})$.
\end{enumerate} 

\begin{prop} \label{S<S'} For $\alpha\ge1$, we have
$$\sup_{z\in \mathcal{S}}\Psi_{\alpha}(z) \ \ \le \ \ \sup_{z\in \mathcal{S}'}\Psi_{\alpha}(z).$$\end{prop}

\begin{proof} Let us start by outlining the structure of the proof.  
Pick an arbitrary $z\in \mathcal{S}$.  We will gradually improve $z$ by a series of subsequent combinatorial reductions
$$z \longrightarrow z^{(1)} \longrightarrow z^{(2)} \longrightarrow z^{(3)} \longrightarrow z^{(4)},$$
such that
\begin{center}$\Psi_{\alpha}(z) \  \le \  \Psi_{\alpha}(z^{(i)}) \  \le \  \Psi_{\alpha}(z^{(j)})$   \ \ \ \   for \ \ \ \ $1\le i\le j\le 4$, \end{center}
and  $z^{(i)}$ will satisfy the requirements from $1.$ to $i.$ in the definition of $\mathcal{S}'$. This will give $\Psi_{\alpha}(z) \le \Psi_{\alpha}(z^{(4)})$ for some  $z^{(4)}\in \mathcal{S}'$ and the claim will be proved.

\smallskip

\medskip

1. $z \rightarrow z^{(1)} $. Put $z= (z_0, z_1, \dots, z_{n+1})$. If $z_{i-1}\not=z_i$ for all $i\in \{1,2,\dots, n+1\}$, then we are done.
Otherwise, let $i_0$ be the smallest index without this property. As $z_0=0$ and $z_1$ is strictly positive, we must have $i_0>1$. Analogously, we have $i_0<n+1$.
Consequently, observe that $z_{i_0-1}$ and $z_{i_0}$ are negligible. Examine the transformation $z\mapsto \tilde{z}$,
\begin{equation}\label{remove/rescale}(\dots, z_{i_0-1}, z_{i_0}, z_{i_0+1}, \dots) \ \longrightarrow \ w^{-1}\cdot (\dots, z_{i_0-1}, z_{i_0+1}, \dots),\end{equation}
$$w \ = \ 1-z_{i_0},$$
which removes $z_{i_0}$ and rescales the remaining elements. If $z_{i_0+1}\in \phi_{\alpha}(z)$, then $w^{-1}z_{i_0+1}$ will remain a significant component of $\tilde{z}$.
The contribution of  $z_{i_0+1}$ (and all the other significant components of $z$) to the overall sum will grow by a factor of $w^{-1}>1$. The contribution of  $z_{i_0-1}$ to $\Psi_{\alpha}(z)$ is zero and it can only increase if  $z_{i_0-1}$ becomes significant. Therefore $\Psi_{\alpha}(z) \le \Psi_{\alpha}(\tilde{z})$. After a finite number of such operations, we obtain a sequence $z^{(1)}$ for which $1.$ holds. \smallskip
\medskip

2. $ z^{(1)} \rightarrow z^{(2)}$. Set $z^{(1)}=(z_i^{(1)})_{i=0}^{n+1}$ and suppose that $(z_{i_0-1}^{(1)}, z_{i_0}^{(1)}, z_{i_0+1}^{(1)})$ 
 is a split for some $i_0\in \{2,3,\dots, n-1\}$ -- by the definition of split configuration, $i_0$ must be greater than $1$ and smaller  than $n$. Accordingly, note that $ z_{i_0}^{(1)}$ is negligible and consider the preliminary modification $z^{(1)}\mapsto \hat{z}^{(1)}$ given by
 $$(\dots, z_{i_0-1}^{(1)}, z_{i_0}^{(1)}, z_{i_0+1}^{(1)}, \dots) \ \longrightarrow \ (\dots, z_{i_0-1}^{(1)},0, z_{i_0+1}^{(1)}, \dots),$$
 which changes $z_{i_0}^{(1)}$ into $0$ (so $\hat{z}^{(1)}\not\in \mathcal{S}$: we will handle this later). 
 As $z_{i_0-1}^{(1)}>z_{i_0}^{(1)}$, we have
 \begin{equation} \label{LEFTzi0} \Bigg|\frac{z_{i_0-1}^{(1)}}{z_{i_0-2}^{(1)}+z_{i_0-1}^{(1)}}-\frac{z_{i_0-1}^{(1)}}{z_{i_0-1}^{(1)}+z_{i_0}^{(1)}}\Bigg| \ \ \ < \ \ \ \Bigg|\frac{z_{i_0-1}^{(1)}}{z_{i_0-2}^{(1)}+z_{i_0-1}^{(1)}}-1\Bigg|,\end{equation}
 if only  ${z_{i_0-1}^{(1)}\in \phi_{\alpha}(z^{(1)})}$.
 Similarly, as $z_{i_0}^{(1)}<z_{i_0+1}^{(1)}$,  we get
 \begin{equation} \label{RIGHTzi0}\Bigg|\frac{z_{i_0+1}^{(1)}}{z_{i_0}^{(1)}+z_{i_0+1}^{(1)}}-\frac{z_{i_0+1}^{(1)}}{z_{i_0+1}^{(1)}+z_{i_0+2}^{(1)}}\Bigg| \ \ < \ \ \ \Bigg|1-\frac{z_{i_0+1}^{(1)}}{z_{i_0+1}^{(1)}+z_{i_0+2}^{(1)}}\Bigg|,\end{equation}
as long as ${z_{i_0+1}^{(1)}\in \phi_{\alpha}(z^{(1)})}$. By (\ref{LEFTzi0}) and (\ref{RIGHTzi0}), with a slight abuse of notation (the domain of $\Psi_{\alpha}$ formally  does not  contain $\hat{z}^{(1)}$, but we may extend the definition for $\Psi_{\alpha}(\hat{z}^{(1)})$ in a straightforward way), we can write $\Psi_{\alpha}(z^{(1)}) \le \Psi_{\alpha}(\hat{z}^{(1)})$. Now, let us denote
$$\hat{z}^{(1, \leftarrow)} \ = \ (0, \hat{z}_{1}^{(1)}, \dots, \hat{z}_{i_0-1}^{(1)}, 0)$$
 and 
 $$\hat{z}^{(1, \rightarrow)} \ = \  (0,  \hat{z}_{i_0+1}^{(1)}, \dots,  \hat{z}_{n}^{(1)}, 0).$$ 
In other words, sequences $\hat{z}^{(1, \leftarrow)}$ and $\hat{z}^{(1, \rightarrow)}$ are two consecutive parts of $\hat{z}^{(1)}$ and we can restore
$\hat{z}^{(1)}$ by glueing their corresponding zeroes together. Moreover, after normalizing them by the weights 
\begin{center}$w^{(1, \leftarrow)} =  \sum_{i=1}^{i_0-1}\hat{z}_{i}^{(1)}$  \ \ \ \ and \ \ \ \ $w^{(1, \rightarrow)}= \sum_{i=i_0+1}^{n}\hat{z}_{i}^{(1)}$, \end{center}
we get \ $(w^{(1, \leftarrow)})^{-1} \hat{z}^{(1, \leftarrow)}, \ (w^{(1, \rightarrow)})^{-1}\hat{z}^{(1, \rightarrow)}\in \mathcal{S}$. \ Next, in this setup, we are left with
\[\begin{aligned}
\Psi_{\alpha}(\hat{z}^{(1)}) \ \ \ =&  \ \ \ w^{(1, \leftarrow)}\cdot \Psi_{\alpha}\left( \frac{\hat{z}^{(1, \leftarrow)}}{w^{(1, \leftarrow)}}\right) \\
+& \ \ \  w^{(1, \rightarrow)}\cdot \Psi_{\alpha}\left(\frac{\hat{z}^{(1, \rightarrow)}}{w^{(1, \rightarrow)}}\right) \\
\le & \ \ \ \max \left\{\Psi_{\alpha}\left( \frac{\hat{z}^{(1, \leftarrow)}}{w^{(1, \leftarrow)}}\right), \Psi_{\alpha}\left(\frac{\hat{z}^{(1, \rightarrow)}}{w^{(1, \rightarrow)}}\right) \right\}, \end{aligned} \]
where we have used $w^{(1, \leftarrow)}+w^{(1, \rightarrow)}=1$. Let
$$\tilde{z}^{(1)} \ = \ \arg\max \left\{\Psi_{\alpha}(z) : \ z \in \left\{ \frac{\hat{z}^{(1, \leftarrow)}}{w^{(1, \leftarrow)}}, \frac{\hat{z}^{(1, \rightarrow)}}{w^{(1, \rightarrow)}}\right\} \right\}.$$
By the construction, we have $\Psi_{\alpha}(z^{(1)}) \le \Psi_{\alpha}(\tilde{z}^{(1)})$,
the new sequence $\tilde{z}^{(1)}$ is shorter than $z^{(1)}$ and $\tilde{z}^{(1)}$ contains less   split configurations than $z^{(1)}$. After repeating this procedure ($z^{(1)}\mapsto \tilde{z}^{(1)}$) multiple times, we acquire a new sequence $z^{(2)}$ obeying $1.$ and $2.$ \smallskip

\medskip

3.  $z^{(2)} \rightarrow z^{(3)}$. Surprisingly, it is enough to put $z^{(3)}=z^{(2)}$. Indeed, we can show that sequence $z^{(2)}$ already satisfies the third condition. First, suppose that $(z_{j_0-1}^{(2)}, z_{j_0}^{(2)}, z_{j_0+1}^{(2)})$ and $(z_{j_1-1}^{(2)}, z_{j_1}^{(2)}, z_{j_1+1}^{(2)})$ are two different peaks with indices $j_0<j_1$. Hence, as $z_{j_0}^{(2)}>z_{j_0+1}^{(2)}$ and $z_{j_1-1}^{(2)}<z_{j_1}^{(2)}$, there is at least one point
$i_0\in\{j_0+1, \dots, j_1-1\}$ at which we are forced to ``flip'' the direction of the previous inequality sign:
$$z_{j_0-1}^{(2)}< z_{j_0}^{(2)}>z_{j_0+1}^{(2)}>\dots > z_{i_0}^{(2)}<\dots<z_{j_1-1}^{(2)}< z_{j_1}^{(2)}>z_{j_1+1}^{(2)}.$$ 
Equivalently, this means that $(z_{i_0-1}^{(2)}, z_{i_0}^{(2)}, z_{i_0+1}^{(2)})$ is a split configuration. This contradicts our initial assumptions about $z^{(2)}$ (the requirement $2.$ is not met) and proves that there is at most one peak in $z^{(2)}$. Second, we have 
\begin{center}$0=z_{0}^{(2)}<z_1^{(2)}$  \ \ \ and \ \ \  $z_{n}^{(2)}>z_{n+1}^{(2)}=0$,\end{center}
so there exists a point $j_0$ at which the direction of the inequalities must be changed from `$<$' to `$>$'. Thus, there is at least one peak in $z^{(2)}$. \smallskip
\medskip

4. $ z^{(3)} \rightarrow z^{(4)}$. Let $z^{(3)}=(z_i^{(3)})_{i=0}^{n+1}$ and assume that  $(z_{j_0-1}^{(3)}, z_{j_0}^{(3)}, z_{j_0+1}^{(3)})$ is the unique peak of $z^{(3)}$:
\begin{equation}\label{1PEAK} 0<z_1^{(3)}<\dots<z_{j_0-1}^{(3)} < z_{j_0}^{(3)} > z_{j_0+1}^{(3)} > \dots > z_n^{(3)}>0.\end{equation}
Further reasoning is similar to the previous ones (from points $1.$ and $2.$), so we will just sketch it.  
If the requirement $4.$ is not satisfied, pick a negligible component $z_{i_0}^{(3)}$ with $i_0\not=j_0$. Next, apply the transformation  $z^{(3)}\mapsto  \tilde{z}^{(3)}$ defined by (\ref{remove/rescale}), i.e. remove $z_{i_0}^{(3)}$ and rescale the remaining components. Thanks to the `single peak structure' (\ref{1PEAK}), all the significant components of $z^{(3)}$ remain significant for $\tilde{z}^{(3)}$. The terms associated with components  $z^{(3)}_i \in \phi_{\alpha}(z^{(3)})\setminus \{z_{i_0-1}^{(3)}, z_{i_0+1}^{(3)} \}$ are not changed (and their contribution grows after the rescaling). The summands corresponding to $z_{i_0-1}^{(3)}$ and  $z_{i_0+1}^{(3)}$  can only increase,  just as in (\ref{LEFTzi0}) and (\ref{RIGHTzi0}).
Therefore $\Psi_{\alpha}(z^{(3)}) \le \Psi_{\alpha}(\tilde{z}^{(3)})$. After several repetitions and discarding of all unnecessary negligible components (beyond the central $z_{j_0}$), we finally obtain the desired sequence $z^{(4)}\in \mathcal{S}'$.
\end{proof}

We proceed to the proof of our main result.

\begin{proof}[Proof of Theorem \ref{2/e}] We start with the lower estimate, for which the argument is simpler. By Proposition \ref{SUP-->PHI} and reformulation (\ref{sup-NEW}), for $\alpha > 2$ we have 
\begin{align*}
\alpha\cdot \sup_{(X,Y)\in \mathcal{C}}\mathbb{E}|X-Y|^{\alpha}  \ \ \ =&  \ \ \  \alpha \cdot \sup_{z\in \mathcal{S}}\Phi_{\alpha}(z) \\
\ge& \ \ \  \alpha \cdot \Phi_{\alpha}\left(0, \frac{1}{\alpha}, \frac{\alpha-2}{\alpha}, \frac{1}{\alpha}, 0\right)  \\
= & \ \ \ \alpha \cdot \frac{2}{\alpha}\left|1-\frac{1}{\alpha-1} \right|^{\alpha} \ \xrightarrow{\alpha \rightarrow \infty} \ \ \frac{2}{e}.  \end{align*} 
Now we turn our attention to the upper estimate. By  Propositions \ref{onlySignificant} and \ref{S<S'}, we get
$$\alpha\cdot \sup_{(X,Y)\in \mathcal{C}}\mathbb{E}|X-Y|^{\alpha}  \ \ \ \le \ \ \ \alpha\cdot\Bigg( \Big|1-\frac{1}{1+\sqrt{\alpha}}\Big|^{\alpha} \  +  \ \sup_{z\in \mathcal{S}'}\Psi_{\alpha}(z)\Bigg).$$
Next, because of
$$\lim_{\alpha \to \infty} \alpha\cdot \Big|1-\frac{1}{1+\sqrt{\alpha}}\Big|^{\alpha} \ = \ 0,$$
it is enough to provide an asymptotic estimate for  $\alpha\cdot\sup_{z\in \mathcal{S}'}\Psi_{\alpha}(z)$. 
 Fix an arbitrary $z= (z_0, z_1,\dots, z_{n+1})\in \mathcal S'$ and let $z_{j_0}$ be the center of the unique peak contained in $z$: 
  $$0<z_1<\dots<z_{j_0-1} < z_{j_0} > z_{j_0+1} > \dots > z_n>0.$$
 As $z_{j_0}$ is the only negligible component  contained in $z$, we have
 \begin{center}$\sqrt{\alpha}\cdot z_i<z_{i+1}$   \ \ \ \   for \ \ \ \ $1\le i\le j_{0}-1$, \end{center}
 and
 \begin{center}$z_{i-1}>\sqrt{\alpha}\cdot z_i$   \ \ \ \   for \ \ \ \ $j_0+1\le i\le n$. \end{center}
 In particular, we get  $0\le z_{j_0-1}, z_{j_0+1} < 1/\sqrt{\alpha}$.  Consequently, we can write $\Psi_{\alpha}(z)  =  A + B + C$, where 
 $$A \  =  \ \sum_{|i-j_0|>2} z_i  \Big| \frac{z_i}{z_{i-1}+z_i}-\frac{z_i}{z_i+z_{i+1}} \Big|^{\alpha},$$
 $$B \  = \  z_{i_0-2}  \Big| \frac{ z_{i_0-2} }{ z_{i_0-3} + z_{i_0-2} }-\frac{ z_{i_0-2} }{ z_{i_0-2} + z_{i_0-1} } \Big|^{\alpha} 
 \ + \ z_{i_0+2}  \Big| \frac{ z_{i_0+2} }{ z_{i_0+1} + z_{i_0+2} }-\frac{ z_{i_0+2} }{ z_{i_0+2} + z_{i_0+3} } \Big|^{\alpha}$$
 and
 $$C \ = \  z_{i_0-1}  \Big| \frac{ z_{i_0-1} }{ z_{i_0-2} + z_{i_0-1} }-\frac{ z_{i_0-1} }{ z_{i_0-1} + z_{i_0} } \Big|^{\alpha} 
 \ + \ z_{i_0+1}  \Big| \frac{ z_{i_0+1} }{ z_{i_0} + z_{i_0+1} }-\frac{ z_{i_0+1} }{ z_{i_0+1} + z_{i_0+2} } \Big|^{\alpha}.$$
 We will examine these three parts separately.

 \smallskip

\noindent\emph{The term $A$.} Since $z_i/(z_{i-1}+z_i)$ and $z_{i}/(z_i+z_{i+1})$ belong to $[0,1]$, we may write
 \begin{align*}
A \ \ \ \le&  \ \ \   \sum_{i=1}^{j_0-3}z_i  \ \ + \ \ \sum_{i=j_0+3}^{n}z_i \\
<& \ \ \ z_{j_0-3}\cdot \sum_{i=0}^{j_0-4} \left(\frac{1}{\sqrt{\alpha}}\right)^i \ \  + \ \ z_{j_0+3}\cdot \sum_{i=0}^{n-j_0-3} \left(\frac{1}{\sqrt{\alpha}}\right)^i  \\
< & \ \ \  (z_{j_0-1}+z_{j_0+1})\cdot \frac{1}{\alpha} \cdot \sum_{i=0}^{\infty} \left(\frac{1}{\sqrt{\alpha}}\right)^i
\\
< & \ \ \  \frac{2}{\alpha\sqrt{\alpha}} \cdot \sum_{i=0}^{\infty} \left(\frac{1}{\sqrt{\alpha}}\right)^i  \ \  =  \ \ \frac{2}{\alpha(\sqrt{\alpha}-1)} \end{align*} 
 and hence
 $$\alpha\cdot A \ < \ \frac{2}{\sqrt{\alpha}-1} \ \xrightarrow{\alpha \rightarrow \infty}  \ 0.$$

\smallskip

\noindent\emph{The term $B$.}  
 We have 
  \begin{align*}
B  \ \ \ \le&  \ \ \  
  z_{i_0-2}  \Big|1-\frac{ z_{i_0-2} }{ z_{i_0-2} + z_{i_0-1} } \Big|^{\alpha} \ \ + \ \  z_{i_0+2}  \Big|\frac{ z_{i_0+2} }{ z_{i_0+1} + z_{i_0+2}}-1\Big|^{\alpha}  \\
<& \ \  z_{i_0-2}  \left|1-\frac{ z_{i_0-2} }{ z_{i_0-2} + \frac{1}{\sqrt{\alpha}} } \right|^{\alpha} \ \ + \ \  z_{i_0+2}  \left|\frac{ z_{i_0+2} }{ \frac{1}{\sqrt{\alpha}} + z_{i_0+2}}-1\right|^{\alpha}  \\
\le & \ \ \ 2\cdot \sup_{x\in [0,1]} x\left|1-\frac{x}{x+\frac{1}{\sqrt{\alpha}}} \right|^{\alpha} \ \ = \ \ \frac{2}{\sqrt{\alpha}(\alpha-1)}\cdot \left(1-\frac{1}{\alpha} \right)^{\alpha}.
 \end{align*} 
 This yields 
  $$\alpha\cdot B \ < \ \frac{2\sqrt{\alpha}}{\alpha-1}\cdot \left(1-\frac{1}{\alpha} \right)^{\alpha} \ \xrightarrow{\alpha \rightarrow \infty}  \ 0.$$
  
\noindent \emph{The term $C$.} Finally, we observe that 
  \begin{align*}
C  \ \ \ \le&  \ \ \  
  z_{i_0-1}  \Big|1-\frac{ z_{i_0-1} }{ z_{i_0-1} + z_{i_0} } \Big|^{\alpha} \ \ + \ \  z_{i_0+1}  \Big|\frac{ z_{i_0+1} }{ z_{i_0} + z_{i_0+1}}-1\Big|^{\alpha}  \\
\le& \ \  z_{i_0-1}  \left|1-z_{i_0-1}   \right|^{\alpha} \ \ + \ \  z_{i_0+1}  \left|z_{i_0+1}-1\right|^{\alpha}  \\
\le & \ \ \ 2\cdot \sup_{x\in [0,1]} x\left|1-x \right|^{\alpha} \ \ = \ \ \frac{2}{\alpha+1}\cdot \left(1-\frac{1}{\alpha+1}\right)^{\alpha}.
 \end{align*} 
 Consequently, we obtain
   $$\alpha\cdot C \ \le \ \frac{2\alpha}{\alpha+1}\cdot \left(1-\frac{1}{\alpha+1} \right)^{\alpha} \ \xrightarrow{\alpha \rightarrow \infty}  \ \frac{2}{e}.$$
   
   \smallskip
   
\noindent The estimates for $A$, $B$ and $C$ give the desired upper bound. The proof is complete.
\end{proof}

\setlength{\baselineskip}{2ex}

\bibliographystyle{plain}
\bibliography{ASYMbib}

\begin{thebibliography}{10}

\bibitem{B4}
I.~Arieli and Y.~Babichenko.
\newblock A population's feasible posterior beliefs.
\newblock {\em Proceedings of the 23rd ACM Conference on Economics and
  Computation}, 2022.

\bibitem{B2}
I.~Arieli, Y.~Babichenko, and F.~Sandomirskiy.
\newblock Persuasion as transportation.
\newblock {\em Proceedings of the 23rd ACM Conference on Economics and
  Computation}, 2022.

\bibitem{B1}
I.~Arieli, Y.~Babichenko, F.~Sandomirskiy, and O.~Tamuz.
\newblock Feasible joint posterior beliefs.
\newblock {\em Journal of Political Economy}, 129, 2021.

\bibitem{GaleRyser}
S.~Boza, M.~K\v{r}epela, and J.~Soria.
\newblock Lorentz and {G}ale--{R}yser theorems on general measure spaces.
\newblock {\em Proceedings of the Royal Society of Edinburgh: Section A
  Mathematics}, 152, 2022.

\bibitem{contra}
K.~Burdzy and S.~Pal.
\newblock Can coherent predictions be contradictory?
\newblock {\em Advances in Applied Probability}, 53, 2021.

\bibitem{pitman}
K.~Burdzy and J.~Pitman.
\newblock Bounds on the probability of radically different opinions.
\newblock {\em Electronic Communications in Probability}, 25, 2020.

\bibitem{mastersthesis}
S.~Cichomski.
\newblock Maximal spread of coherent distributions: a geometric and
  combinatorial perspective.
\newblock Master's thesis, University of Warsaw, 2020.
\newblock available at arXiv:2007.08022 [math.PR].

\bibitem{EJP}
S.~Cichomski and A.~Os\k{e}kowski.
\newblock The maximal difference among expert's opinions.
\newblock {\em Electronic Journal of Probability}, 26, 2021.

\bibitem{contra2}
S.~Cichomski and A.~Os\k{e}kowski.
\newblock Contradictory predictions with multiple agents.
\newblock 2022.
\newblock (preprint) available at arXiv:2211.02446 [math.PR].

\bibitem{kDoob}
S.~Cichomski and A.~Os\k{e}kowski.
\newblock Doob's estimate for coherent random variables and maximal operators
  on trees.
\newblock 2022.
\newblock (preprint) available at arXiv:2211.02434 [math.PR].

\bibitem{BPC}
S.~Cichomski and F.~Petrov.
\newblock A combinatorial proof of the {B}urdzy--{P}itman conjecture.
\newblock {\em Electronic Communications in Probability}, 28, 2023.

\bibitem{C1}
A.~P. Dawid, M.~H. DeGroot, and J.~Mortera.
\newblock Coherent combination of experts' opinions.
\newblock {\em Test}, 4, 1995.

\bibitem{C2}
M.~H. DeGroot.
\newblock A bayesian view of assessing uncertainty and comparing expert
  opinion.
\newblock {\em Journal of Statistical Planning and Inference}, 20, 1988.

\bibitem{extmart}
L.~E. Dubins and G.~Schwarz.
\newblock On extremal martingale distributions.
\newblock 1967.
\newblock Proceedings of the Fifth Berkeley Symposium on Mathematical
  Statistics and Probability.

\bibitem{B3}
K.~He, F.~Sandomirskiy, and O.~Tamuz.
\newblock Private private information.
\newblock 2022.
\newblock Proceedings of the 23rd ACM Conference on Economics and Computation.

\bibitem{EDSM}
K.~Hestir and S.~C. Williams.
\newblock Supports of doubly stochastic measures.
\newblock {\em Bernoulli}, 1, 1995.

\bibitem{C4}
R.~Ranjan and T.~Gneiting.
\newblock Combining probability forecasts.
\newblock {\em Journal of the Royal Statistical Society: Series B (Statistical
  Methodology)}, 72, 2010.

\bibitem{rudin}
W.~Rudin.
\newblock {\em Functional analysis. Second edition}.
\newblock McGraw-Hill, 1991.

\bibitem{C3}
V.~A. Satop{\"a}{\"a}, R.~Pemantle, and L.~H. Ungar.
\newblock Modeling probability forecasts via information diversity.
\newblock {\em Journal of the American Statistical Association}, 111, 2016.

\bibitem{tao}
T.~Tao.
\newblock Szemer\'edi's regularity lemma revisited.
\newblock {\em Contributions to Discrete Mathematics}, 1, 2006.

\bibitem{zhu}
T.~Zhu.
\newblock {\em Some Problems on the Convex Geometry of Probability Measures}.
\newblock PhD thesis, UC Berkeley, 2022.
\newblock available at https://escholarship.org/uc/item/8001f519.

\end{thebibliography}

\end{document}